\documentclass[12pt]{amsart}
\usepackage{amsmath, amsfonts, latexsym, array, amssymb, }
\usepackage{psfrag}
\usepackage[all]{xy}
\usepackage{graphicx}
\usepackage{cancel}
\usepackage{hyperref}

\usepackage[usenames,dvipsnames]{color}

\newcommand{\jb}[1]{{\textcolor{NavyBlue}{#1}}} 
\newcommand{\JB}{\color{NavyBlue}}

\usepackage{enumerate} 

\numberwithin{equation}{section}
\theoremstyle{plain} 

\newtheorem{thm}{Theorem}[section]
\newtheorem{theorem}[equation]{Theorem}
\newtheorem{cor}[equation]{Corollary} 
\newtheorem{corollary}[equation]{Corollary}
\newtheorem{lem}[equation]{Lemma}
\newtheorem{lemma}[equation]{Lemma}
\newtheorem{fact}[equation]{Fact}

\newtheorem{prop}[equation]{Proposition}
\newtheorem{proposition}[equation]{Proposition}

\newtheorem{claim}[equation]{Claim}

\newtheorem{que}[equation]{Question}
\newtheorem{pbm}[equation]{Problem}

\newtheorem{summary}[equation]{Summary}

\theoremstyle{remark} 

\newtheorem{remark}[equation]{Remark}
\newtheorem{remarks}[equation]{Remarks}

\theoremstyle{definition} 
\newtheorem{defn}[equation]{Definition}
\newtheorem{definition}[equation]{Definition}

\newcommand{\annotation}[1]{\marginpar{\tiny #1}}
\newcommand{\question}[1]{\medskip\noindent{\bf Question.} #1\medskip}
\newcommand\chk[1]{{\bf\Large (CHECK: #1)}}
\newcommand{\comment}[1]{}

\newcommand{\new}[1]{{\bf #1 }}
\newcommand{\caution}[1]{\medbreak\noindent{\sc CAUTION:} #1\medbreak}

\newcommand\alter[1]{\left\{\begin{array}{l}#1\end{array}\right.}
\newcommand\sd{\bigtriangleup} 
\newcommand\dotcup{{\mathaccent\cdot\cup}}

\newcommand\mme{m.m.e.}
\newcommand\pe{\mathbb P_{\textnormal{erg}}}
\newcommand\Proberg{\mathbb P_{\textnormal{erg}}}
\newcommand\pen{\mathbb P_{\textnormal{erg}}'} 
\newcommand\eps{\epsilon}
\newcommand\id{\operatorname{id}}
\newcommand\NN{{\mathbb N}}
\newcommand\RR{{\mathbb R}}
\newcommand\TT{{\mathbb T}}
\newcommand\ZZ{{\mathbb Z}}
\newcommand\pih{{\hat\pi}}
\newcommand\PP{{\mathbb P}}
\newcommand\Prob{{\operatorname{Prob}}}
\newcommand{\sarig}{Sarig regular}
\newcommand\Sigmah{{\hat\Sigma}}
\newcommand\sprat{\sigma_{\operatorname{rat}}}
\newcommand\diffsym{\operatorname{\Delta}}

\newcommand\xxk{\widetilde{X}^k}
\newcommand\xxq{X^{(q)}}

\newcommand\ov{\overline}
\newcommand\hXN{\hat{X}_N}
\newcommand\ret{{{\textnormal{ret}}}}
\newcommand\topo{{\operatorname{top}}}
\newcommand\nullset{\emptyset}
\newcommand\uep{union-entropy-period}
\newcommand\UEP{u.e.p.}

\newcommand\Thm{Thm.~}
\newcommand\Lem{Lem.~}
\newcommand\Prop{Prop.~}
\newcommand\Fact{Fact~}
\newcommand\Cor{Cor.~}

\newcommand\suppress[1]{}

\begin{document}

\begin{title}[Almost Borel Structure]
{The almost Borel structure of surface diffeomorphisms, Markov shifts and their factors}
\end{title}

\author{Mike Boyle}
\address{Department of Mathematics - University of Maryland}
\email{mmb@math.umd.edu}
\author{J\'er\^ome Buzzi}
\address{Laboratoire de Math\'ematiques d'Orsay - Universit\'e Paris-Sud}
\email{jerome.buzzi@math.u-psud.fr}


\dedicatory{Dedicated to Roy Adler, in appreciation} 

\begin{abstract}
Extending work of Hochman, we study the almost-Borel structure, i.e., 
the nonatomic invariant probability measures, of symbolic systems and surface diffeomorphisms. 

We first classify Markov shifts and characterize them as strictly universal with respect to a natural family of classes of Borel systems. We then study their continuous factors showing that a low entropy part is almost-Borel isomorphic to a Markov shift but that the remaining part is much more diverse, even for finite-to-one factors. However, we exhibit a new condition which we call 
\lq Bowen type\rq\ which gives complete control of those factors.

This last result applies to and was motivated by the symbolic covers of Sarig. We find complete numeric invariants for Borel isomorphism of $C^{1+}$ surface diffeomorphisms modulo zero entropy measures; for those admitting a totally ergodic measure of positive (not necessarily maximal) entropy, we get a classification up to almost-Borel isomorphism.
\end{abstract} 
\maketitle

\setcounter{tocdepth}{1}
\tableofcontents 

\section{Introduction} 

Much of the richness of dynamical systems theory comes from
understanding systems with respect to different structures (smooth,
measurable, etc.). In this paper we are interested in the almost-Borel 
structure of surface diffeomorphisms. More precisely we study them as automorphisms of standard Borel spaces up to sets negligible for all invariant, nonatomic
Borel probability measures, following Hochman
\cite{Hochman} (see also \cite{VaradErgDec}). 

We analyze  Markov shifts (generalizing \cite{Hochman} to the 
non-irreducible, non-mixing case) and especially their factors, both under continuous and what we call Bowen type factor maps. 
We finally show that this applies to Sarig's symbolic dynamics \cite{Sarig} of surface diffeomorphisms.

\subsection{Surface diffeomorphisms}
We consider
surface diffeomorphisms which are  $C^{1+}$ smooth, i.e.,  
with H\"older continuous derivative. 
(We refer to  Sec.\ \ref{sec.back} for definitions and background.)
Our main result, \Thm\ref{surfacetheorem}, implies:

\begin{thm}\label{mainthm.conjug}
Any  $C^{1+}$-diffeomorphism  of a compact surface is Borel isomorphic to a countable state Markov shift, up to a subset negligible with respect to all ergodic measures\footnote{By measure we will  
(outside Appendix \ref{sec:BPD}) always mean invariant Borel probability measure.} with positive entropy.
\end{thm}

We will deduce a classification involving the periods of 
ergodic mea\-su\-re-pre\-ser\-ving 
systems $(S,\mu )$ defined as follows. Recall that the rational spectrum is:
 \begin{equation}\label{eq.sprat}
   \sprat(S,\mu):=\{e^{2i\pi r}:r\in\mathbb Q,\;\exists f\in L^2(\mu),\;
 f\circ S=e^{2i\pi r} f 
\text{ and } f\neq 0\}.
 \end{equation}
A positive integer $p$ is a \emph{period} if $e^{2i\pi/p}\in\sprat(S,\mu)$.
In Sec.\ \ref{sec.proofSurfClass}, we will prove the following,  using a classification of Markov shifts (\Thm\ref{mainthm.classif} below):

\begin{thm}\label{mainthm.surfclass}
Two $C^{1+}$-diffeomorphisms  of compact surfaces are Borel isomorphic, up to a subset negligible with respect to all ergodic measures with positive entropy, if and only if the following data are equal for both: for each $p\geq1$,
 \begin{enumerate}
  \item the supremum of the positive entropies of ergodic measures which have a maximum period that is equal to $p$;
 \item if this supremum is positive, the cardinality of the set of
   non\-atomic ergodic measures that achieves 
the previous supremum.
 \end{enumerate}
\end{thm}


\subsection{Almost-Borel classification and Markov shifts}
We need the generalization  to the non-mixing case of the characterization and classification of
Markov shifts obtained by Hochman \cite{Hochman}.

First some definitions. An automorphism of a standard Borel space is a \emph{Borel
   system} (see Sec.\ \ref{sec.borelsyst}). We denote by $\pen(S)$ its
 set of ergodic, nonatomic measures.

\begin{definition}
Two Borel systems $(X,S)$ and $(Y,T)$ are \emph{almost-Borel isomorphic} if there exists a Borel 
isomorphism  $\psi:X'\to Y'$ with invariant 
Borel subsets $X'\subset X$ and $Y'\subset Y$ such that:
 \begin{itemize}
 \item $\psi\circ S=T\circ\psi$ on $X'$;
 \item $X\setminus X'$ and $Y\setminus Y'$ are \emph{almost null sets}: $\mu(X\setminus X')=\nu(Y\setminus Y')=0$ for all  $\mu\in\pen(S)$ and $\nu\in\pen(T)$.
  \end{itemize}
\end{definition}

Thus two systems are almost-Borel isomorphic if, in the terminology of \cite{Hochman}, their free parts are Borel isomorphic on full sets. 
We refer to the discussion in \cite[p. 394]{Weiss1} for a comparison
with Borel and  measurable isomorphisms.


Let $T$ be a Markov shift (a \lq\lq subshift of finite type over a countable alphabet\rq\rq, see Sec.\ \ref{sec.back.markov} for this and related definitions). Up to an almost null set, it is a disjoint, at most countable, union of irreducible Markov shifts $T_i$, $i\in I$, not reduced to periodic orbits.  Throughout this paper,
all Markov shifts satisfy:
 \begin{equation}\label{eq.fec0}
  \text{all irreducible components have finite entropy.}
 \end{equation}

For each $T_i$, let $p_i$ be its period,  $h_i>0$ be its  entropy and set $m_i=1$ or $0$ according to whether $T_i$ has or not a nonatomic measure of entropy $h_i$.
Define  two sequences 
over $\NN:=\{1,2,\dots\}$:
 \begin{equation}\label{eq.u-eta}\begin{aligned}
    &\bar u_T(p):=\sup\biggl(\{h_i:i\in I,\; p_i|p\}\cup\{0\}\biggr)\in[0,\infty]
    \text{ and }\\
    &\bar\eta_T(p):=\sum\{m_i:i\in I,\; (h_i,p_i)=(\bar u_T(p),p)\}\in\{0,1,\dots,\infty\}.
 \end{aligned}\end{equation}

We can now state the extension of  Hochman's classification proved in Sec.\ \ref{sec.markov.class}:

\begin{theorem}\label{mainthm.classif}
Two Markov shifts $S,T$ are almost-Borel isomorphic if and only if $(\bar u_T,\bar \eta_T)=(\bar u_S,\bar \eta_S)$.
Moreover, sequences $u,\eta$ coincide with sequences $\bar u_T,\bar \eta_T$ of some Markov shift $T$ if and only if 
 \begin{equation}\label{cond.MS}
    \forall p\geq1\quad
    u(p)=\sup_{q|p} u(q)\text{ and }u(p)=\infty\implies \eta(p)=0.
 \end{equation}
\end{theorem}

In  Sec.\ \ref{sec.markov.char}, we find a \lq\lq maximal Markov subsystem\rq\rq\ inside an arbitrary Borel system:

\begin{theorem}\label{mainthm.markovpart}
Any Borel system $(X,S)$ contains an invariant Borel subset $X_U$ such that:
 \begin{enumerate}
  \item $X_U$ is
    almost-Borel 
isomorphic to a Markov shift $T$ with $\bar
    \eta_T\equiv0$;
  \item if some subsystem $Y\subset X$ satisfies the previous property, then $Y\setminus X_U$ is almost null.
 \end{enumerate}
These two properties define $X_U$ up to an almost null set.
\end{theorem}

The condition \lq\lq$\bar \eta_T\equiv0$\rq\rq\ cannot be removed:
consider
the product of a positive entropy shift of finite type 
with the identity map on the unit interval.
This condition and the above result is very natural from the point of view of universality discussed in Sec.\ \ref{s.univ-heur}.

\medbreak

This leads to a characterization of Markov shifts up to almost Borel isomorphism.
We say 
that a measure-preserving system $(S,\mu)$ is \emph{$p$-Bernoulli}
($p\in\NN$) if it is isomorphic to the product of a Bernoulli system
and a circular permutation on $p$ points.\footnote{Note, $p$ is the
  maximum period of $(S,\mu)$ in the terminology of Theorem
  \ref{mainthm.surfclass}.} We call it \emph{periodic-Bernoulli} if we don't want to specify $p$. At the end of Sec. \ref{sec.markov.class}, we prove:

\begin{corollary}\label{cor.manySFTs}
A Borel system $(X,S)$ is almost-Borel isomorphic to a Markov shift if
and only if 
there is a sequence $u:\NN\to[0,\infty]$ with $u(p)=\max_{q|p}u(q)$ such that:
 \begin{enumerate}
  \item for each $p\in\NN$ and $t< u(p)$, there is an almost-Borel embedding of an irreducible Markov shift of period $p$ and entropy $>t$ into $X$;
  \item the set $\mathcal M$ of ergodic measures $\mu\in\pen(S)$ such that for every period $p$ of $\mu$, $h(S,\mu)\geq u(p)$, is at most countable;
  \item each $\mu\in\mathcal M$ is $p$-Bernoulli for some $p\in\NN$ and $h(S,\mu)=u(p)$. 
 \end{enumerate}
\end{corollary}

The mixing case was analyzed by Hochman (see \cite[Thm. 1.7]{Hochman}
and the discussion that precedes  it).

\begin{remark} \label{hyperremark}
This characterization provides an alternate 
approach to results like Theorem \ref{mainthm.conjug} by
splitting the dynamics between: a \lq\lq top entropy part" which must be shown to carry only
very specific measures; and the rest which carries all possible
measures  \lq\lq below some entropy thresholds".  
If $S$ is a $C^{1+}$ diffeomorphism of a compact manifold and 
$S$ has no zero Lyapunov exponents, then this second part can be analyzed using 
Katok's horseshoes 
(see \cite{BuzziLausanne}).
\end{remark}

\subsection{Factors of Markov shifts}
Thus we are led to find  
conditions guaranteeing that a dynamical system 
has shifts of finite type as large (in entropy) subsystems. 
There is an interest of some vintage in this problem 
 (e.g. \cite{Katok1980, Marcus1985, Petersen1986}).
In Sec.\ \ref{sec.continuousFactorsMS}, we prove 

\begin{theorem}\label{thm.continuousfactorMS}
Let $(X,S)$ be an irreducible Markov shift with period $p$ and let
$\pi:(X,S)\to(Y,T)$ be a continuous, not necessarily surjective, factor map  into a selfhomeomorphism of a Polish space. Let
 $$
    h_*(\pi):=\sup\{ h(T,\pi_*\mu):\mu\in\Proberg(T)\}.
 $$
For any $h<h_*(\pi)$, there is an irreducible shift of finite type
$X'\subset X$ such 
that $h_\topo(X')>h$, $X'$ has period $p$, and the restriction of 
$\pi$ to $X'$ is injective. 
\end{theorem}

Without additional assumptions, $\pi(X)$ can carry measures  
with entropy $> h_*(\pi)$ and unrelated to those of $X$ (see Prop.\ \ref{countablebad}).
Even when $X$ is compact and $h_*(\pi)=h_\topo(\pi(X))=h_\topo(X)$, the \mme's, that is, the \emph{ergodic measures maximizing entropy} for $\pi(X)$, do not have to be images of \mme's of $X$. In fact, we show that they can include uncountably many copies of measures which are not periodic-Bernoulli (Cor.\ \ref{sfttopcor}).

Next 
we assume $\pi$ to be finite-to-one, continuous and with compact image. This forces $h_*(\pi)=h_\topo(\pi(X))=h_\topo(X)$ and the \mme's of $\pi(X)$ to be finitely many periodic-Bernoulli measures. However, 
the \emph{periodic-maximal} measures, i.e.,  the measures maximizing
the entropy among measures \emph{with a given period} can
still be more or less arbitrary (see Cor.\ \ref{badfinite}), in
contrast
to those of Markov shifts. To control this, we use the following property.

\begin{definition}\label{def.AlmostBowen}
 Let $\pi: (X,S) \to  (Y,T)$ 
be a Borel factor map from a Markov shift 
into a Borel system, and $B$ an invariant 
Borel subset of $X$.  
 Then $\pi$ 
is {\it Bowen type}  on $B$ (or relative to $B$) if  
there is a relation $\sim$ 
on the alphabet of $X$ such 
that
the following hold: 
\begin{enumerate} 
\item 
$\pi (x) = \pi (w) \iff x\sim w $, for all $x,w$ in $B$ ,  
\text{and}
\item 
$ x\sim w \, \implies \, \pi (x) = \pi (w)$, 
for all $x,w$ in 
$X$ ,
\end{enumerate}  
where $x\sim w$ means $x_n\sim w_n$ for all $n$. 
If $B=X$, one simply says that $\pi$ is Bowen type.
\end{definition} 

This definition is adapted from a property 
pointed out by Bowen 
\cite[p.13]{Bowenaxioma} 
for surjective continuous factor maps from
shifts of finite type 
to systems associated with Markov partitions. 
More precisely, these factors  are 
David Fried's  {\it finitely presented dynamical 
systems} \cite{Fisher,Friedfp}; these are the expansive systems 
which are continuous factors of shifts of finite type.   

For a Markov shift $Z$, the \emph{\sarig\ set} $Z_{\pm\ret}$ of
  $Z$ is the 
subset of sequences in which some symbol 
appears infinitely often in the past and some symbol 
(not necessarily the same) appears infinitely often in the future. In Sec. \ref{sec.bowen} we prove:

\begin{theorem} \label{bowentypefactors} 
Suppose $(X,S)$  is a 
 Markov shift satisfying  condition \eqref{eq.fec0}
 and 
$\pi: (X,S)\to (Y,T)$ is a Borel 
factor map such that, 
for each irreducible component $Z$ of $X$,
\begin{enumerate}   
\item 
$\pi$ is Bowen type on the \sarig\ set $Z_{\pm\text{ret}}$, and 
\item 
the restriction $\pi|Z_{\pm ret}$ is finite-to-one.
\end{enumerate} 
Then, letting $\bar X$ be the union of the \sarig\ sets $Z_{\pm\ret}$ of the irreducible components $Z$ of $X$,
 \begin{itemize}
  \item $\pi(\bar X)\subset Y$ is almost-Borel isomorphic to a 
Markov shift;
  \item the induced map  $\pen (\bar X) \to \pen (\pi \bar X)$ is surjective. 
\end{itemize}

\end{theorem} 

Condition (1) above is really about the restrictions $\pi|Z$.

\medbreak

In Sec.\ \ref{sec:surface} we shall apply this theorem to Sarig's symbolic dynamics and deduce \Thm\ref{surfacetheorem} from which Theorems \ref{mainthm.conjug} and \ref{mainthm.surfclass} follow.

\subsection{The universality heuristic}\label{s.univ-heur}

A Borel system $X$ is \emph{universal} with respect to a class
$\mathcal C$ of Borel systems, if any system in $\mathcal C$ can be
almost-Borel embedded into $X$. If, additionally, $X$ belongs to
$\mathcal C$, it is said to be \emph{strictly universal}. Strictly
universal systems, when they exist, are unique up to almost-Borel
isomorphism. In this case, universal systems can be characterized as
unions of an essentially unique 
\lq\lq maximal\rq\rq  strictly universal 
system and a complementary part (see Sec.\ \ref{sec.univ}).

Hochman showed that many systems of entropy 
$h$ are $h$-universal, i.e., universal with respect to the class of Borel systems whose measures have entropy $<h$ (see \Thm\ref{thm.Hochman} and \Prop\ref{p.UnivMark}). The  complementary system mentioned above  
then supports exactly the ergodic measures of entropy $h$, often a unique measure of maximum entropy which 
is Bernoulli. 

This provides a general heuristic: in a suitable 
class of systems, for a suitable notion of ``universal'', 
analyze each system as the union 
of a (large) standard universal part and 
a complementary part (hopefully managable). This approach 
 gives our almost 
Borel results on $C^{1+}$ surface diffeomorphisms 
and Markov shifts, with Hochman's universality refined
to address periods.  The details of this universality approach are 
spelled out in Sections \ref{sec.univ} and \ref{sec:markov}. 

The existence of a large universal part can be rather robust. 
For example, any continuous factor $Y$ of a 
mixing shift of finite type is $h(Y)$-universal (by  
\Thm\ref{prop.continuousfactorSFT}). A related result holds for 
continuous factors of Markov shifts 
(\Thm\ref{thm.continuousfactorMS}). 
In contrast, as indicated earlier, 
the possibilities for the complementary system in $Y$ 
can vary wildly without stronger assumptions (see Sec. \ref{pathology}).

\subsection*{Acknowledgments}
We thank David Fried, Jean-Paul Thouvenot and 
Benjamin Weiss for background 
and references for Bowen's work, the weak Pinsker property and 
the theory of Bernoulli shifts. We also thank B. Weiss for 
referring us to the paper \cite{Kieffer-Rahe} of Kieffer and Rahe, 
on which we rely in Appendix \ref{sec:BPD}. 
M. Boyle gratefully acknowledges the support 
 during this work of the Danish National Research Foundation 
through the Centre for Symmetry and Deformation (DNRF92)
and the hospitality and support of the D\'epartement de Math\'ematiques at Orsay (Universit\'e Paris-Sud).

 This paper is dedicated to Roy Adler, coinventor of
topological entropy
\cite{AdlerKonheimMcandrew}, with gratitude for his kindness
and in appreciation of his mathematical influence. 
This paper considers entropy and period for the almost 
Borel classification of Markov shifts; the seminal  
result of this type was the Adler-Marcus Theorem 
\cite{AdlerMarcus}, 
which classified irreducible 
shifts of finite type up to almost topological conjugacy 
by topological entropy and period.

\section{Definitions and background}\label{sec.back}

We fix notations and recall some facts that we will use without further explanation.

\subsection{Dynamical Systems}
In this paper, a \emph{dynamical system} (or system) $S$ is an automorphism of a space $X$. We shall consider: 
 \begin{itemize}
  \item[-] \emph{topological} dynamical systems (or t.d.s.) given by selfhomeomorphisms of (not necessarily compact) metrizable spaces; 
   \item[-] \emph{measure-preserving} systems given by automorphisms of probability spaces. We shall often abbreviate ergodic mea\-su\-re-pre\-ser\-ving systems, to \emph{ergodic} systems;
  \item[-] \emph{Borel} systems given by Borel automorphisms of standard Borel spaces (see below).
 \end{itemize}
Recall that a \emph{factor map}, resp.\ an \emph{embedding}, is a homomorphism, resp.\ a monomorphism, of the spaces that intertwines the automorphisms. Unless a factor map is said to be \emph{into}, it is assumed to be surjective.  A \emph{subsystem} is a system of the same category given by a restriction to an invariant subspace.

We often use the symbol for the space or for the automorphism
to refer to the system 
and its domain and suppress the structure (topological, Borel,\dots) from the notation, with interpretation 
by context. 

\subsection{Borel spaces} \label{borelsubsec}
A \emph{standard Borel space} \cite[Sec.\ 12]{Kechris} is a set $X$ together with a $\sigma$-algebra $\mathcal X$ generated by a Polish topology, i.e., a topology defined by some distance which turns $X$ into a separable, complete, metric space.
The elements of $\mathcal X$ are called the \emph{Borel sets} of
$X$. 

$f:X\to Y$ is a \emph{Borel map} if
$X$ and $Y$ are standard Borel spaces and the preimage of any Borel subset is Borel. $f$ is a \emph{Borel isomorphism} if it is a bijection such that $f$ and $f^{-1}$ are Borel. Here, no sets are considered negligible. According to Kuratowski's theorem (see \cite[(15.6)]{Kechris}),
all uncountable standard Borel spaces are isomorphic. 

Recall that if $f:X\to Y$ is a Borel map and $A$ is a Borel subset of $X$ such
that $f|A$ is injective, then $f(A)$ is Borel and $f:A\to f(A)$ is a
Borel isomorphism, according to
the Lusin-Souslin Theorem \cite[(15.2)]{Kechris}.

We denote by $\Prob(X)$ the set of
not necessarily invariant probability measures defined over the Borel sets. We endow it with the $\sigma$-algebra generated by the maps
$\mu\mapsto \mu(E)$, $E\in\mathcal X$.
 This makes $\Prob(X)$ into a standard Borel space (see \cite[(17.24)]{Kechris} and 
\cite[beginning of section 17.E]{Kechris}). 

%


\subsection{Almost-Borel systems}\label{sec.borelsyst} Let $(X,S)$ be a Borel system. Then $\Prob(S)\subset\Prob(X)$ is the set of $S$-invariant Borel probability measures of $X$ (henceforth the measures of $S$) and $\Proberg(S)$ is the subset of ergodic invariant measures. $\Prob(S)$ and $\Proberg(S)$ are Borel subsets of $\Prob(X)$, hence they also are standard Borel spaces.

 

An {\it almost null set} for $(X,S)$ is a Borel 
set of measure zero for every $\mu$ in $\pen (S)$, the set of atomless, ergodic measures of $S$. 
By an \emph{almost-Borel system}, we mean a Borel system up to an almost null set. An {\it almost-Borel  map} means a homomorphism of Borel systems  defined on the complement of an almost null set. Almost-Borel embeddings, factors, and isomorphisms are defined in the obvious way. 

We shall need the following Borel maps (see, e.g.,
\cite{BuzziLausanne}), defined on the complement of an almost null set: (1) a map $M:X\to\Proberg(S)$ such that, for
any Borel set $B\subset\Proberg(S)$ and  any $\mu\in\Proberg(S)$:
$\mu(M^{-1}(B))=1$ if and only if $\mu\in B$;\footnote{For compact
  t.d.s.,
we can take
$M(x)=\lim_{n\to\infty}\frac1n\sum_{k=0}^{n-1}\delta_{S^kx}$,  
defined on the Borel set of points for which this weak star limit exists.} 
(2) the map $h:\Prob(S)\to[0,\infty]$ associating to each measure its Kolmogorov-Sinai entropy (see below).

The following almost-Borel variant of the well-known measurable
Schr\"oder-Bernstein theorem \cite[(15.7)]{Kechris} is fundamental
  for us :

\begin{proposition}[Hochman \cite{Hochman}]\label{lem:HCB}
Two Borel systems are almost-Borel isomorphic if and only if there are almost-Borel embeddings of one into the other. 
\end{proposition}

\subsection{Entropy} 
\label{sec:entropy} 
The \emph{topological entropy} of a compact t.d.s. $(Y,T)$ is denoted by $h_\topo(T)$. The \emph{Kolmogorov-Sinai entropy}  of a measure-preserving system $(S,\mu)$ is denoted by $h(S,\mu)$. We define the \emph{Borel entropy} of a Borel system $(X,S)$ to be $h(S):=\sup\{h(S,\mu):\mu\in\Prob(S)\}$. We shall often call any of these the entropy of $T$, $(S,\mu)$ or $S$.


The \emph{variational principle for entropy} states that if $(Y,T)$ is a compact t.d.s., its Borel entropy $h(T)$ coincides with its \emph{topological entropy} $h_\topo(T)$.
An  \emph{ergodic measure of maximum entropy} (or \mme) for 
 $(X,S)$ is a measure $\mu\in\Proberg(S)$ such that $h(S,\mu)=h(S)$. It does not need to exist or be unique, even for compact t.d.s.

We will use the Bowen-Dinaburg formulas to compute $h_\topo(T)$ in terms of  \emph{dynamical $(\epsilon,n) $-balls} 
$
B(p,\epsilon,n ) = \{ y\in Y: 0\leq k < n \implies 
\text{dist}(T^kp,T^ky) < \epsilon \}\ .
$ 
Recall the following for a compact subset $C$ of $Y$ and $\epsilon >0$.
The integer $r_\text{span}(\eps,n,C,T)$ is the minimal cardinality of $(\eps,n)$-spanning sets for $C$ and 
$r_{\text{sep}}(\epsilon, n, C, T)$ is the maximal cardinality of an 
$(\epsilon,n)$-separated subset of $C$.  
We have 
\begin{equation} 
\label{eq.BowenFormulaCover}
\begin{aligned} 
h_{\text{sep}} (C,T,\eps) \ &\ := 
       \limsup_{n\to \infty }
\frac 1n \log r_{\text{sep}}(\epsilon, n, C, T) \ , \\ 
h_{\text{span}} (C,T,\eps) \ &\ := 
       \limsup_{n\to \infty }
\frac 1n \log r_{\text{span}}(\epsilon, n, C, T) \ , \quad
\text{and}\\  
h_\topo(Y)\  &=\  \lim_{\epsilon \to 0} h_{\text{sep}} (Y,T,\eps) \ =  \ 
\lim_{\epsilon \to 0} h_{\text{span}} (Y,T,\eps) \ . 
\end{aligned}
\end{equation} 

We refer to \cite{KHbook,Petersenbook,Waltersbook} for more background.

\subsection{Markov shifts}\label{sec.back.markov}
  A \emph{countable state Markov shift} (or just Markov shift) is $(X,S)$ where $X\subset V^\ZZ$ for some countable (maybe finite) set $V$ and for some $E\subset V^2$:
 $
    X = \{ x\in V^\ZZ:\forall n\in\ZZ\; (x_n,x_{n+1})\in E\}
 $
and $S:X\to X$ defined by
 $
    S((x_n)_{n\in\ZZ})=(x_{n+1})_{n\in\ZZ}.
 $
The directed graph $G=(V,E)$ is a \emph{vertex} presentation of $(X,S)$. The distance $d(x,y)=\exp\left(-\inf\{|k|:x_k\ne y_k\}\right)$ turns $X$ into a separable, complete metric space and $S$ into a homeomorphism.

A finite or infinite sequence $x=(x_i)_{i\in I}$ is a \emph{path} on
the graph $G$ if $I\subset\ZZ$ is an interval, each $x_i\in V$ and
each $(x_i,x_{i+1})\in E$  whenever $\{i,i+1\}\subset I$. The length $|x|$ of $x$ is the cardinality of $I$. If $|x|<\infty$, then we call it a word and define the
\emph{cylinder}: 
$[x]_X$ (or just $[x]$) to be 
$\{y\in X:\forall i\in I\; x_i=y_i\}$. 

If $x\in X$ and $a\leq b$ are two integers, $x|_a^b$ \text{is} the word $x_ax_{a+1}\dots x_{b-1}$ of length $b-a$. A \emph{loop} of length $n$ based at a vertex $v$  is a finite word $\ell_0\dots\ell_{n-1}$ such that $\ell_0=v$ and $\ell_0\dots\ell_{n-1}\ell_0$ is a path on $G$. We note that the Gurevi\v{c} entropy (see \cite{Gurevic1970}) of a Markov shift, defined in terms of its loops, is equal to its Borel entropy.

The classical \emph{shifts of finite type} (or SFTs) are the topological dynamical systems topologically isomorphic to a compact Markov shift, or equivalently, to a Markov shift that can be presented by a finite graph. We refer to \cite{LindMarcus1995} for background.

The Markov shift $(X,S)$ is \emph{irreducible} if it can be presented by a strongly connected graph $G$, i.e., such that any two vertices $u,v$ can be joined by a path from $u$ to $v$. In this case, its \emph{period} is the greatest common divisor of the lengths of all loops on $G$. 
$(X,S)$ is \emph{mixing} if it is irreducible with period $1$. 

Any Markov shift $(X,S)$ can be written as the disjoint union of  irreducible Markov shifts $(X_j,S_j)$, $j\in J$ with $J$ countable (possibly finite), and a set of measure zero with respect to any invariant measure. This decomposition is unique (up to the indexing) and the Markov subshifts $(X_j,S_j)$, $j\in J$, are called the \emph{irreducible components} of $(X,S)$.
 
On an irreducible period $p$ Markov
shift 
$(X,S)$ with
finite Borel entropy, the measure of maximal
entropy (or \mme), if it exists, is unique and 
{\it $p$-Bernoulli}. Moreover:

\begin{fact}\label{f.irred-hp}
For any $h\in(0,\infty)$ and $p\in\NN$, one can find two irreducible
Markov shifts with entropy $h$ and period $p$: one with a measure of
maximum entropy, one without. 
\end{fact}


Finally, we note that from a directed graph $G=(V,E)$ 
(now possibly with multiple edges from one vertex to another)
one has also the {\it edge shift} associated to $G$. This is a Markov 
shift whose alphabet is the set of edges of $G$. 
In terms of the earlier definition, the edge shift of $G$ 
is defined by a new graph $G'$, whose vertex set is $E$, in which
there is an edge from $e_1$ to $e_2$ iff the terminal vertex in $G$ of 
$e_1$ equals the initial vertex in $G$ of $e_2$. We will use the edge shift 
presentation in Sec.\ \ref{pathology}.  
We refer to \cite{Kitchens1998} for more background on Markov shifts.

\subsection{Periods of measures and Borel decomposition}
\label{periodssubsec} 


Let $(S,\mu)$  be an ergodic system. Recall the notion of periods from eq.\ \eqref{eq.sprat}. Note that if $p$ is a period, then any positive divisor of $p$ is also a period and that $p$ is a period iff there is a $p$-cyclic partition modulo $\mu$, i.e., $\{X_0,X_1,\dots,X_{p-1}\}\subset\mathcal X$ such that $\mu(\bigcup_{i=0,\dots,p-1} X_i)=1$ and $\mu(X_i\cap X_j)=0$ for all $0\leq i\ne j\leq p-1$.

Observe that not every measure has a maximum period (consider
odometers). 
If it exists, then the set of all periods is the set of divisors of the maximum period. Also having maximum period equal to $1$ is equivalent to 
$\sprat(S,\mu)=\{1\}$ and 
(because $(S,\mu)$ is ergodic) it is equivalent to total ergodicity (i.e., the ergodicity of all $(S^n,\mu)$, $n\geq1$).

\begin{fact}\label{f.MSmaxper}
Given an irreducible Markov shift $X$ with period $p$ and entropy $h$, the supremum of the entropies of ergodic measures with maximum period $p$ is equal to $h$. Conversely, for any ergodic invariant measure carried by $X$, the maximum period, if it exists, is a multiple of $p$.
\end{fact}

In the above definitions, the partition is relative to $\mu$. It is important for our purposes that we can improve this as follows.


\begin{theorem}[Borel periodic decomposition]\label{thm:specdec}
Let $(X,T)$ be an automorphism of a standard Borel space. For each
integer $p\geq1$, there exists a Borel partition
$P(p):=\{P_1,\dots,P_{p},P_*\}$ of $X$ such that:
 \begin{itemize}
   \item $T(P_*)=P_*$ and $T(P_i)=P_{i+1}$ for all $i=1,\dots,p$ ($P_{p+1}:=P_1$);
 \item for any $\mu\in\Proberg(T)$, $\mu(P_*)=0$ if and only if $p$ is a period of $(S,\mu)$.
 \end{itemize}
\end{theorem}
Though related results exist (see \cite[remark on top of page 399]{Weiss1}), we could not find this statement in the literature, hence a proof is given in Appendix \ref{sec:BPD}.

\section{Universal systems}\label{sec.univ}

We study Markov shifts as almost-Borel systems. In this section, we perform the part of the analysis that is conveniently done in the language of universality (already used by Hochman \cite{Hochman}, following Benjamin Weiss, e.g., \cite{Weiss2}).  

\begin{definition}
Let $\mathcal C$ be a class of almost-Borel systems. An almost-Borel
system $(X,S)$ is \emph{$\mathcal C$-universal} if it contains (the
image of) an almost-Borel embedding of any system in $\mathcal C$. If,
additionally, $(X,S)\in\mathcal C$, then it is said to be
\emph{strictly} $\mathcal C$-universal.\footnote{This is related
  to  but distinct from the notion of a terminal object in category theory.} 
\end{definition}

We build and classify \lq\lq maximal universal parts\rq\rq\ of arbitrary almost-Borel systems. The next section will relate these to Markov shifts by appealing to Hochman's theorem \cite{Hochman}.

\subsection{Period-universal systems}

Following \Prop\ref{lem:HCB}, \lq the\rq\ strictly universal system with respect to a given class, if it exists, is unique up to almost-Borel isomorphism. Hochman identified the strictly universal systems with respect to the classes  $\mathcal B(t)$, $t\geq0$, of Borel systems $(X,S)$ such that for all $\mu\in\pen(S)$, $h(S,\mu)<t$. 

We consider for each $t\geq0$ and $p\in\NN$, the class $\mathcal B(t,p)$ of systems whose measures $\mu\in\pen(S)$ satisfy: $p$ is a period and $h(S,\mu)<t$. For short we write that a system is \emph{$t$-universal}, resp.\ \emph{$(t,p)$-universal} if it is $\mathcal B(t)$-universal, resp.\ $\mathcal B(t,p)$-universal.
We will repeatedly use (see \Prop1.4(3) of \cite{Hochman} in the case $p=1$ --its proof generalizes):

\begin{lemma}\label{lem.unionSUp}
For $p\in \NN$  and $h\in[0,\infty]$, a countable union of strictly 
$(h_n,p)$-universal 
systems, is strictly $(h,p)$-universal with $h=\sup h_n$.  
\end{lemma}

The following almost-Borel invariant is important for Markov shifts and related systems.

\begin{definition}\label{def.useq}
The \emph{{\rm(}\uep{\,\rm)} universality sequence}  of an almost-Borel system $(X,S)$ is $u_S:\NN\to[0,\infty]$ defined by:
 $$
   u_S(p):=\sup\{t\geq0:(X,S)\text{ contains a strictly $(t,p)$-universal system}\}.
 $$
\end{definition}

\begin{remarks}
 Prop. \ref{p.UnivMark} will show that strictly $(t,p)$-universal systems do exist hence the above invariant is not trivial and can be computed as $u_S(p)=\sup\{t\geq0:(X,S)$ is $(t,p)$-universal$\}$. Also, $u_S(p)$ does not need to be the supremum of the entropies of measures with a period $p$. 
\end{remarks}

Observe that if $q$ divides $p$, $\mathcal B(t,q)\supset\mathcal B(t,p)$ so $(t,q)$-universality implies $(t,p)$-universality. Hence:

\begin{fact} \label{useqremark}
For all $p\in \NN$, $u_S(p) = \max_{q|p} u_S(q)$. 
 \end{fact} 

A condition defines a set \emph{up to an almost null set} if the symmetric difference between any two Borel subsets satisfying it, is an almost null set.

\begin{prop}\label{uppartexists} 
A Borel system $(X,S)$ contains, for each $p\in \NN$, a subsystem 
$(X_{Up},S_{Up})$ characterized up to an almost null set by the two following equivalent properties.

\noindent (1) For all $\mu \in \pen (S)$: 
\begin{equation}\label{eq.Xup}
\mu X_{Up}=1 \iff  
p \text{ is a period of }\mu \text{ and }h(S,\mu )< u_S(p) \ . 
\end{equation}

\noindent (2) $(X_{Up},S_{Up})$ is a strictly $p$-universal subsystem and contains any other strictly $p$-universal subsystem of $X$ up to an almost null set.

Moreover, $(X_{Up},S_{Up})$ is strictly $(u_S(p),p)$-universal.
\end{prop}

\begin{proof}
Conditions (1) and (2) separately imply uniqueness up to an almost null set so
it suffices to build a solution $(X_{Up},S_{Up})$ to  (1) and check that it satisfies also (2) and the last claim.

\Thm\ref{thm:specdec} gives Borel subsystems $C_p$, $p\geq1$, 
such that for any $\mu\in\pen(S)$, 
$\mu(C_p)=1$ if and only if $p$ is a period of $\mu$. 
Recall that the functions $M(\cdot)$ and $h(S,\cdot)$ from Sec.\
\ref{borelsubsec} and Sec.\ \ref{sec:entropy} are Borel.  Hence for any $t\in(0,\infty]$
there is an invariant Borel subset $V^t$ of $X$  
such that, for all $\mu\in\pen(S)$, $\mu(V^t)=1$ if and only if $h(S,\mu)<t$. 
Set $X_{Up}= C_p \cap V^t$ with $t=u_S(p)$. 
Clearly $X_{Up}$ is a solution to (1). 

We turn to condition (2).
First, $(X_{Up},S_{Up})$ is strictly $(u_S(p),p)$-universal by 
Lemma \ref{lem.unionSUp}. Second, if $X'\subset X$ is strictly $p$-universal, then it must be $(t,p)$-universal with $t\leq u_S(p)$. Thus $X'\subset X_{Up}$ up to an almost null set by \eqref{eq.Xup}. (2) and the last claim are satisfied. 
\end{proof}

\subsection{Union-entropy-period universal parts} 

The following class of Borel systems will help us analyze 
not necessarily irreducible Markov shifts.

\begin{definition}
For a sequence $u:\NN\to[0,\infty]$,  $\mathcal C(u)$ denotes the \uep\ class of Borel systems $(X,S)$ such that any $\mu\in\pen(S)$ has some period $p$  such that $h(S,\mu)< u(p)$. A \emph{strictly \UEP-universal system} is a strictly $\mathcal C(u)$-universal system for some $u:\NN\to[0,\infty]$. 
\end{definition}

Considering the subsystems $X_p:=C_p\cap V^{u(p)}$ as in the proof of Proposition \ref{uppartexists} easily yields:

\begin{fact}\label{fact.uep}
For any $u:\NN\to[0,\infty]$,
 $(X,S)\in\mathcal C(u)$ if and only if $X=\bigcup_{p\in\NN} X_p$ with $X_p\in B(u(p),p)$ for all $p\in\NN$.
If $X$ is strictly $\mathcal C(u)$-universal, then each $X_p$ is strictly $(u(p),p)$-universal.
\end{fact}


An arbitrary Borel system $(X,S)$  contains a \lq maximal\rq\ strictly \UEP-universal subsystem:

\begin{theorem} \label{univpartexists}
For any Borel system $(X,S)$ satisfying:
  \begin{equation}\label{eq.fec}
    \forall \mu\in\pen(S)\; h(S,\mu)<\infty,
 \end{equation}
there is a subsystem $(X_U,S_U)$ characterized up to an almost null set by each of the following three equivalent properties.\\
\noindent (1)  $X_U=\bigcup_{p\in\NN} X_{Up}$ up to an almost null set.

\noindent (2)  For all  $\mu \in \pen (S)$,  
\begin{equation}\label{eq.uni}
\mu X_{U}=1 \iff  
\mu \text{ has a period }p\text{ s.t. }
h(S,\mu )< u_S(p) \ .
\end{equation}

\noindent (3) $(X_U,S_U)$ is a strictly \UEP-universal subsystem that contains any strictly \UEP-universal subsystem up to an almost null set.

Moreover, $(X_U,S_U)$ is strictly $\mathcal C(u_S)$-universal and its universality sequence coincides with $u_S$.
\end{theorem} 

\begin{definition}\label{def.ueppart}
The subsystem $(X_U,S_U)$ above is called the (\uep) \emph{universal part}  of $(X,S)$.  
\end{definition}

The following are easy consequences of universality.

\begin{corollary} \label{useqtheorem}
Suppose $(X,S)$ and $(Y,T)$ are Borel systems.  Then 
\begin{enumerate} 
\item 
There is an almost-Borel embedding $(X_U,S_U)\to (Y_U,T_U)$ 
if and only if $u_S\leq u_T$. 
\item 
$(X_U,S_U)$ and $(Y_U,T_U)$ are almost-Borel isomorphic 
if and only if $u_S = u_T$. 
\item \label{absorbing}
Suppose for all $\mu\in \pen(X)$ there is a period  $p$ 
of $\mu$ such that $h(S, \mu)<u_T(p)$. 
Then the systems $(X,S) \cup (Y,T)$, $(X,S) \sqcup (Y,T)$,  and $(Y,T)$ are almost-Borel 
isomorphic.
\end{enumerate} 
\end{corollary}

The proof of \Thm\ref{univpartexists} relies on the following
lemma, whose proof we defer to the end of the section. Say that a
Borel system $(X,S)$ is \emph{stable} if there is an almost-Borel
embedding of $(X\times\{0,1,\dots\},S\times\id)$ into $(X,S)$. Note that the
strictly universal systems with respect to $\mathcal B(t)$, $\mathcal
B(t,p)$, or $\mathcal C(u)$, are stable. Moreover, countable 
unions of stable systems are stable.

\begin{lemma}\label{lem.makedisjoint}
A countable union $\bigcup_{n\geq0} X_n$ of stable subsystems 
is almost-Borel isomorphic to the corresponding disjoint union $\bigsqcup_{n\geq0} X_n$.
\end{lemma} 

\begin{proof}[Proof of \Thm\ref{univpartexists}]
Each of the conditions (1), (2) and (3) implies uni\-que\-ness up to an almost null set. It suffices to show that $X_U$ as in condition (1) with $S_U:=S|X_U$ satisfies the two other claims.
For Claim (2) this follows from Condition (1) of Prop.\ \ref{uppartexists}.

To prove the universality stated in Claim (3), let $(Y,T)\in\mathcal C(u_S)$. By Fact \ref{fact.uep},  $Y=\bigcup_{p\in\NN} Y_p$ with  $Y_p\in\mathcal B(u_S(p),p)$ and $Z_p:=Y_p\setminus\bigcup_{q<p}Y_q$, $p\in\NN$, is a partition. By Prop.\ \ref{uppartexists}, each $X_{Up}$ is strictly $(u_S(p),p)$-universal so there is an almost-Borel embedding of $Z_p\subset Y_p$ into $X_{Up}$ for all $p\geq1$. Now, Lemma \ref{lem.makedisjoint} lets us assume that $X_U=\bigcup_{p\in\NN} X_{Up}$ is a partition, proving $\mathcal C(u_S)$-universality. It is strict since $(X_U,S_U)\in\mathcal C(u_S)$ by Claim (2).

For the second half of (3), let $(Y,T)$ be a strictly $\mathcal C(v)$-universal subsystem of $(X,S)$ for some $v:\NN\to[0,\infty]$. Fact \ref{fact.uep} implies $Y=\bigcup_{p\in\NN} Y_p$ and $v\leq u_S$. By Prop.\ \ref{uppartexists}, $Y_p\subset X_{Up}\cup N_p$ for some almost null $N_p$: $Y\subset X_U\cup\bigcup_{p\in\NN} N_p$ and Claim (3) follows.

Finally, let $u_U$ be the universality sequence of $(X_U,S_U)$. As $X_U\subset X$, $u_U\leq u_S$. The converse inequality follows from the strict universality of each $X_{Up}$.
\end{proof}

\begin{proof}[Proof of Lemma \ref{lem.makedisjoint}] 
It suffices to build an almost Borel embedding $\Psi:\bigcup_{n\geq0} X_n\times\{n\}\hookrightarrow\bigcup_{n\geq1} X_n$ (the reverse embedding is obvious and the lemma then follows from Prop.\ \ref{lem:HCB}).  
We claim that there exist subsystems $Z_0,Z_1,\dots$ such that:
 \begin{enumerate}
  \item each set $Z_n\subset X_0\cup\dots\cup X_n$ is almost Borel isomorphic to $X_n$;
  \item $\phi_n$ is an almost-Borel embedding of $X_n\times\{0,1,\dots\}$ into $Z_n$;
  \item the sets $\phi_\ell(X_\ell\times\{\ell\})$, $0\leq \ell<n$, are pairwise disjoint.
  \item $Z_n\cap\phi_\ell(X_\ell\times\{\ell,n+1,n+2,\dots\})=\emptyset$ for $0\leq \ell<n$.
\end{enumerate}    
Then,  $\Psi:\bigcup_{n\geq0} X_n\times\{n\}\hookrightarrow \bigcup_{n\geq0} X_n$ defined by $\Psi(x,n)=\phi_n(x,n)$ proves the lemma.

We proceed by induction. To begin with, let $\phi_0:X_0\times\{0,1,\dots\}\hookrightarrow Z_0:= X_0$ be given by the stability assumption. Properties $(1)_0,(2)_0,(3)_0,(4)_0$ (i.e., (1),\dots,(4) for $n$ taking the value $0$) are satisfied. 

For $n\geq1$, we assume $(1)_m,(2)_m,(3)_m,(4)_m$ for $0\leq m<n$ and, letting $\tilde X_k:=X_k\setminus(X_0\cup\dots\cup X_{k-1})$, we set:
 \begin{equation} \label{eq.Zn}
       Z_n:=\tilde X_n\cup \bigcup_{k=0}^{n-1}\phi_k((\tilde X_k\cap X_n)\times\{n\}).
  \end{equation}
First note that, using $(1)_k$ for $k<n$, $Z_n\subset \tilde X_n\cup\bigcup_{k<n} X_k\subset \bigcup_{k\leq n} X_k$.
Second we check that the union in \eqref{eq.Zn} is disjoint. Note, $\tilde X_n\cap\phi_k(X_k\times\{0,1,\dots\})\subset\tilde X_n\cap (X_0\cup\dots\cup X_k)=\emptyset$ for $0\leq k<n$. So it is enough to note that for all $0\leq \ell<k<n$, $(4)_{k}$ yields:
 $$
   \phi_\ell((\tilde X_\ell\cap X_n)\times\{n\})\cap \phi_k((\tilde X_k\cap X_n)\times\{n\}) \subset \phi_\ell(X_\ell\times\{k+(n-k)\})\cap Z_k = \emptyset.
 $$
The disjointness in \eqref{eq.Zn} implies that $Z_n$ is isomorphic to $X_n$ so $(1)_n$ holds. Moreover, the stability assumption gives $\phi_n$ as in condition $(2)_n$. 

We prove $(4)_n$ for $0\leq \ell<n$. We use \eqref{eq.Zn} to expand $Z_n$. As before $\tilde X_n\cap Z_\ell=\emptyset$ so we need only to show that, for $0\leq k<n$:
 \begin{equation}\label{eq.disj0}
    \phi_k(X_n\times\{n\})\cap \phi_\ell(X_\ell\times\{\ell,n+1,n+2,\dots\})=\emptyset.
  \end{equation}
If $\ell=k$, \eqref{eq.disj0} follows from the injectivity of $\phi_k$. If  $\ell<k$, it follows from $(4)_k$ as $\phi_k(X_k\times\{0,1,\dots\})\subset Z_k$ and $\{\ell,k+1,k+2,\dots\}\supset\{\ell,n+1,n+2,\dots\}$. If $k<\ell$, it follows from $(4)_\ell$ using $\phi_\ell(X_\ell\times\{0,1,\dots\})\subset Z_\ell$ and $n\geq\ell+1$.

 \eqref{eq.disj0} and therefore condition $(4)_n$ are established. 
Eq.\ \eqref{eq.disj0} also implies condition $(3)_n$, completing the inductive step. 
\end{proof}

\section{Finite entropy Markov shifts} \label{sec:markov}

In this section, we prove Theorems \ref{mainthm.classif} and \ref{mainthm.markovpart} as well as Corollary \ref{cor.manySFTs} by relating the universal parts studied in Sec.\ \ref{sec.univ} to Markov shifts using the work of Hochman \cite{Hochman}.
%
%

\subsection{Markov shifts and universality}
As in 
\cite{Hochman}, for $h\geq 0$ the {\it $h$-slice of } $(X,S)$ is  
 a Borel subsystem which, for $\mu \in \pen(X)$, has $\mu$ 
measure 1 if and only if $h(S, \mu )>h$. ``The'' $h$-slice subsystem  
is unique up to an almost null set.
Note that the $0$-slice is an almost null set and that
a system $(X,S)$ with no measure of maximum 
entropy,  is equal to its $h(S)$-slice up to an almost null set. 
We recall the main result of \cite{Hochman}:

\begin{theorem}[Hochman \cite{Hochman}]\label{thm.Hochman}
Let $0\leq t< h$. Any  mixing SFT $X$ with entropy $h$ is $h$-universal.
In particular,  the \emph{$t$-slice}  of $X$ is strictly $t$-universal. 
\end{theorem}

\begin{proposition}\label{p.UnivMark} 
For $p\in \NN$  and $h\in[0,\infty]$,  
the following systems are strictly $(h,p)$-universal (and therefore isomorphic).
 \begin{enumerate}
\item
$h$-slices of irreducible period $p$, entropy $h$ Markov shifts. 
 \item
Irreducible Markov shifts with period $p$ and entropy $h$ 
with no measure of maximal entropy (which exist exactly when
$h<\infty$). 
 \item Countable unions of period $p$ irreducible Markov shifts
with   entropies strictly less than $h$ and with supremum equal to
$h$. 
\end{enumerate}
\end{proposition}

\begin{proof} 
All of this is in Hochman's work for the case $p=1$ 
(see Theorems 1.5 and 1.6, Proposition 1.4 
in \cite{Hochman}). The remark about almost-Borel isomorphism follows from Prop.\ \ref{lem:HCB}.
For $p>1$, observe that a Borel system $(X,S)$ is $(h,p)$-universal if it
 contains a cyclically moving subset  with a period $p$ such 
that the restriction of $S^p$ to it is $h(S^p)$-universal. 
\end{proof}

Recall the notions of $p$-maximal and $p$-Bernoulli
measures (see before Cor. \ref{cor.manySFTs}).

\begin{lemma}\label{l.analysis}
An irreducible Markov shift $(X,S)$ with entropy $h$ and period $p$ satisfying \eqref{eq.fec} has $h<\infty$ and is the disjoint union of a strictly $(h(S),p)$-universal 
system and a system supporting at most one measure from 
$\pen (S)$, which if it exists is the unique measure 
of maximal entropy of $S$, a $p$-Bernoulli measure.
\end{lemma} 

\begin{proof} 
(This follows the proof of \cite{Hochman} for $p=1$.) 
The $h(S)$-slice of $(X,S)$ is 
strictly $(h(S),p)$-universal
(Prop. \ref{p.UnivMark}). There is at 
most one measure of maximum entropy 
\cite{Gurevic1970}, which if it exists is 
a countable state Markov chain, 
and therefore $p$-Bernoulli 
(by \cite{OrnsteinShields} for $p=1$ and then for 
general $p$ by the argument 
of \cite{AdlerShieldsSmorodinsky})
and is supported on the complement of the $h(S)$-slice. 
\end{proof} 

\subsection{Characterizing Markov shifts}\label{sec.markov.char}

Recall 
that $(X_U,S_U)$ is the universal part of
$(X,S)$ (\Thm \ref{univpartexists}) 
and that $u_S:\NN\to[0,\infty]$ is the universality
sequence (Def.\ \ref{def.useq}).   

\begin{theorem}\label{thm.MarkovChar}
Let $(X,S)$ be a Borel system satisfying the finite entropy condition \eqref{eq.fec}. Then the following are equivalent:
\begin{enumerate}
 \item $(X,S)$ is almost-Borel isomorphic to a Markov shift. 
 \item $\pen(X\setminus X_U)$ is at most countable and each $\mu\in\pen(X\setminus X_U)$ is $p$-Bernoulli with entropy equal to $u_S(p)<\infty$ for some $p\in\NN$.
\end{enumerate}
\end{theorem}

It will be convenient to define $\Prob(p)$ as the collection of $p$-Bernoulli measures carried by $X\setminus X_U$ and let 
 \begin{equation}\label{eq.etaS}
   \eta_S(p):=\#\Prob(p).
 \end{equation}

\begin{proof}
First, let $(X,S)$ be a Markov shift. It is a countable union
$\bigcup_{i\in I} X_i$ where each $X_i$ is an irreducible Markov shift
with period 
$p_i$ and  entropy $h_i$.

Applying \Lem\ref{l.analysis}, we get $h_i<\infty$ and $X_i=X'_i\sqcup X''_i$ where $X'_i$ is strictly $(h_i,p_i)$-universal and $X''_i$ is either empty or carries a $p_i$-Bernoulli measure of entropy $h_i$ (and no other measure). Therefore the universal part of $X$ contains $\bigcup_{i\in I} X'_i$. Hence $X\setminus X_U$ carries at most the previous countably many periodic-Bernoulli measures. The period $p$ and entropy $h$ of any periodic-Bernoulli measure not carried by $X_U$ must satisfy $h=h_i\geq u_S(p)$ whenever $p_i=p$ (see \Thm\ref{univpartexists}). But $u_S(p)\geq h_i$ whenever $p_i=p$. Hence $h=u_S(p)$. This proves (1)$\implies$(2).

Conversely, let $(X,S)$ be a Borel system as in (2).
By \Thm\ref{univpartexists}, $X_U=\bigcup_{p\in\NN} X_{Up}$.  
According to \Lem\ref{lem.makedisjoint}, this is almost-Borel isomorphic to a disjoint union $\bigsqcup_{p\in\NN} V_p$ of some strictly $(h_p,p)$-universal systems $V_p$. By \Prop\ref{p.UnivMark}, each $V_p$ is isomorphic to a Markov shift.

Let $p\in\NN$. Each $\mu\in\pen(X\setminus X_U)$ is a periodic-Bernoulli measure. 
Pick an irreducible Markov shift $W_\mu$ with the same period $p$ and entropy $h=u_S(p)$ as $\mu$. 
Now $X$ is almost-Borel isomorphic
to the Markov shift
 $\bigsqcup_{p\in\NN} V_p\sqcup\bigsqcup_{\mu\in\Prob(p)} W_\mu$.
\end{proof}

This implies (note that \Lem\ref{lem.makedisjoint} does not apply):
\begin{corollary} \label{markovunion} 
If $X$ is the 
 (not necessarily disjoint) union of countably many 
systems $X_n$, each of which is almost-Borel isomorphic to 
a Markov shift satisfying \eqref{eq.fec},
then $X$ is almost-Borel isomorphic  to a Markov shift, itself
satisfying \eqref{eq.fec}.
%
\end{corollary}

We now relate Markov shifts with strictly \UEP-universal systems.

\begin{lemma}\label{lem.u-eta}
For a Markov shift, the conditions \eqref{eq.fec0} and \eqref{eq.fec}
are equivalent. For a Borel system $(X,S)$, the sequences $\bar
u_S,\bar\eta_S$ 
and $u_S,\eta_S$ {\rm(}from \eqref{eq.u-eta}, Def.\ \ref{def.useq}, \eqref{eq.etaS}{\rm)}
coincide. Moreover, the following are equivalent:
\begin{enumerate}
\item $(X,S)$ is strictly \UEP-universal;
\item $(X,S)$ is almost-Borel isomorphic to a Markov shift with $\bar\eta_S\equiv 0$.
\end{enumerate}
\end{lemma}

\begin{proof}
We write $X=\bigcup_{i\in I}X_i$ with $p_i, h_i$ as in 
\eqref{eq.u-eta}.
Any $\mu\in\pen(S)$ is carried by some $X_i$ by ergodicity. The equivalence of \eqref{eq.fec0} and \eqref{eq.fec} follows.  Prop.\ \ref{p.UnivMark} implies $u_S\geq\bar u_S$ and $u_S(p)>\bar u_S(p)$ would give a measure with maximum period $p$ and entropy $>\bar u_S(p)$. $\eta_S\equiv\bar\eta_S$ follows from \Thm\ref{thm.MarkovChar}.

Point (3) of \Thm\ref{univpartexists} shows that a Borel system is strictly \UEP-universal if and only if it coincides with its universal part. \Thm\ref{thm.MarkovChar} shows that this is equivalent to condition (2) above.
\end{proof}

Given Lemma \ref{lem.u-eta}, Theorem \ref{mainthm.markovpart} is equivalent to \Thm\ref{univpartexists}.

\subsection{Classification of Markov shifts}\label{sec.markov.class}


\begin{proof}[Proof of  \Thm\ref{mainthm.classif}]
The sequences $u_S,\eta_S$ coincides with $\bar u_S,\bar \eta_S$ according to \Lem\ref{lem.u-eta}.
Clearly the former are 
invariants of almost-Borel isomorphism. To see that these are complete, let $(X,S)$ and $(Y,T)$ be two Markov shifts satisfying \eqref{eq.fec0} and $(u_S,\eta_S)\equiv(u_T,\eta_T)$. 
By Cor.\ \ref{useqtheorem}, $S_U$ and $T_U$ are almost-Borel
isomorphic. 
By \Thm\ref{thm.MarkovChar}, $X\setminus X_U$
carries only periodic-Bernoulli measures. Let $p\in\NN$. Using the
periodic decomposition \Thm\ref{thm:specdec}, one finds a Borel
subset $X^{(p)}\subset X\setminus X_U$ carrying exactly the
$p$-Bernoulli measures of $X\setminus X_U$. Those measures have entropy $u_S(p)$ by \Thm\ref{thm.MarkovChar}. Hence the almost-Borel
isomorphism class 
of $X^{(p)}$ is 
defined by $(p,u_S(p),\eta_S(p))$. 
To conclude, remark that $X\setminus U=\bigsqcup_{p\in\NN} X^{(p)}$ up to an almost null set.

We turn to Claim \eqref{cond.MS}. The necessity of its first half follows from \Fact\ref{useqremark},  while  its second half is a consequence of the finite entropy condition \eqref{eq.fec}.
Conversely, given $(u,\eta)$ satisfying \eqref{cond.MS}, let us build a Markov shift
$(X,S)$ realizing these invariants. 

First, let $X':=\bigcup_{p\in\NN, u(p)>0} V_p$ with $V_p$ a strictly
$(u(p),p)$-universal Markov shift 
(\Prop\ref{p.UnivMark}).
By \Fact\ref{useqremark}, 
$u_S(p)=\sup_{q|p} u_S(q)$, which is $u(p)$. 
Second, let $X'':=\bigcup_{p\in\NN, \eta(p)>0} W_p\times  1_{\eta(p)}$ where $W_p$ is an irreducible Markov shift of entropy $u(p)$
and period $p$ with exactly one measure of maximum entropy and 
$1_{\eta(p)}$ is the identity on a set of cardinality $\eta(p)$. This is possible as $\eta(p)>0$ only if $u(p)<\infty$ (\Lem\ref{l.analysis}).
The Markov shift $X'\cup X''$ satisfies $u_S=u$ and $\eta_S=\eta$.
\end{proof}

\begin{proof}[Proof of Cor.\ \ref{cor.manySFTs}]
For $(X,S)$ almost-Borel isomorphic to a Markov shift $T$, let
$u:=u_S$ 
its universal sequence.  Prop.\ \ref{p.UnivMark} implies Claim
(1). The set $\mathcal M$ defined in Claim (2)  is contained in
$\pen(X\setminus X_U)$ and \Thm\ref{thm.MarkovChar} implies (2) and (3).

Conversely, let $(X,S)$ be a Borel system satisfying conditions
(1)-(3) for some $u:\NN\to[0,\infty]$. (1) implies $u_S\geq u$
and therefore $\mathcal M\subset\pen(X\setminus X_U)$. If
$u(p)>u_S(p)$, $\mathcal M$ would be uncountable. Finally, (2)-(3)
with $u=u_S$ imply condition (2) of \Thm  \ref{thm.MarkovChar} so $X$
is almost-Borel isomorphic to a Markov shift. 
\end{proof}

\section{Continuous factors of Markov shifts:  universality}\label{sec.continuousFactorsMS}

We prove Theorem \ref{thm.continuousfactorMS}. We first deal with the
following compact case and then 
reduce the general case to this one
through an entropy formula.




\begin{theorem} 
\label{prop.continuousfactorSFT}
Let $(X,S)$ be an irreducible SFT with period $p$ and let
 $\pi:(X,S)\to(Y,T)$ be a continuous factor map.
Then, for any $0\leq h<h(T)$, there is a period $p$, irreducible SFT
$X'\subset X$ such 
that $h(X')>h$ and the restriction of 
$\pi$ to $X'$ is injective. 
In particular, $(Y,T)$ is $(h(T),p)$-universal.
\end{theorem} 

\begin{remark} \label{badborel}
The universality claim of
\Thm\ref{prop.continuousfactorSFT} fails badly 
for Borel factor maps, even if finite to one. For example, 
from a mixing shift of finite type with entropy $h>0$,  with 
 the Borel Periodic Decomposition one can show that 
there is a Borel at most 2-to-1 map which collapses 
all ergodic measures with maximum period 2 to ones with maximum period 1, and is the identity on supports of other 
ergodic measures. The image is not $h$-universal. 
\end{remark}

To prove \Thm\ref{prop.continuousfactorSFT},
we will use the formulas 
\eqref{eq.BowenFormulaCover}
for the topological entropy of a t.d.s. in terms of separated and spanning sets. Sec.\ \ref{sec.back.markov} recalls some standard definitions and notations for Markov shifts including $[w]_X$, $[w]$, $|w|$, $x|_a^b$, and $|w|$. 

If $v,w$ are two finite words over the alphabet of $X$, 
then $|v|,|w|$ are
their lengths and $[v.w]:=\sigma^{|v|}[v]\cap [w]$ is the cylinder
$\{x\in X:x|_{-|v|}^0=v$ and 
$x|_0^{|w|}=w\}$.
We define $v^\infty.w^\infty$ as the unique point in 
all $[v^n.w^n]$ for $n\geq1$ and $v^\infty:=v^\infty.v^\infty$. We shall write $v$ for its length, e.g., $\ell^A=A|\ell|$.

\begin{proof}[Proof of \Thm\ref{prop.continuousfactorSFT}]
Observe that the claim about universality follows 
immediately from the embedding claim according 
to Proposition \ref{p.UnivMark}. 
We assume $h(T)>0$ (otherwise 
the claim is
trivial). 
Let $G$ be a
strongly connected, finite graph presenting $X$.
Fix $0<\zeta<1$ small enough and then
$h'$ such that $h<(1-\zeta)h(Y)<h'<h(Y)$.
Let $\eta_1 >0$ 
small enough such that 
the separation entropy at scale $4\eta_1$ satisfies 
$h_{\text{sep}}(T,\pi(X),4\eta_1)> h' > h$. 
 Observe that
 \begin{equation}\label{eq.htopoeta}
   h_{\text{sep}}(T,\pi(X),4\eta_1)
=\sup_{v\in G} h_{\text{sep}}(T,\pi([v]),4\eta_1)
   \ .
 \end{equation}
$G$ is finite, hence this supremum is 
achieved at some vertex $v$, which we will denote by $0$: 
\begin{equation} \label{eta_1}
h_{\text{sep}}(T,\pi([0]),4\eta_1)> h' > h \ .
\end{equation} 


\medbreak\noindent{\bf Claim 1.} {\it 
Let 
  $\ell$ and $ \tilde\ell$ be  loops in $G$ 
based at  vertex $0$
such that 
$P:=\pi(\ell^\infty)\ne\tilde P:=\pi(\tilde\ell^\infty)$.
 Then there are  a positive multiple $M$ of $p$ and  a number 
$0<\eta<\eta_1$ such that for all 
integers $A,C\geq M$, 
if  $x,y\in\pi([\ell^A. \tilde\ell^C])$ and 
$-\ell^A+M   \leq k \leq  \tilde\ell^C -M$, then }
 \begin{equation}\label{local.synchro}
      k=0\iff \max_{0\leq j\leq \ell^A-M} d(T^{-j}x,T^{k-j}y)<\eta.
 \end{equation}
{\it Moreover, for any $x,y\in X$ with $x|_{-M}^M=y|_{-M}^M$,
$d(\pi(x),\pi(y))<\eta/4.$ }
\begin{proof}[Proof of Claim 1]
Let $Z=\pi(\ell^\infty.\tilde\ell^\infty)$. As $Z$ is a heteroclinic
point, its orbit is discrete. Define 
 $r_0=\min(d(Z,\mathcal O(Z)\setminus\{Z\}),\eta_1)>0$. 
The uniform continuity of $\pi$ gives $M\in p\NN$ such that, 
for all $u,v\in X$, $u|_{-M}^{M}=v|_{-M}^{M}$ implies
$d(\pi(u),\pi(v))<r_0/16$. 
We will prove Claim 1 for this $M$ and $\eta = r_0/4$. 

Let $\hat x,\hat y\in[\ell^A.\tilde\ell^C]$, 
$x=\pi(\hat x),y=\pi(\hat y)$ and 
$-\ell^A+M   \leq k \leq  \tilde\ell^C -M$. 
Note, $\hat x|_{-\ell^A}^{M}=\hat y|_{-\ell^A}^M$ so, if $k=0$:
$$
0\leq j \leq \ell^A-M \ \implies \ 
 d(T^{-j}x,T^{k-j}y) < r_0/16 = \eta / 4\ . 
$$ 
Also,
$\hat y|_{k-M}^{k+M}=(\ell^\infty.\tilde\ell^\infty) |_{k-M}^{k+M}$, 
so  $d(T^{k}y,T^{k}Z)<r_0/16$ and, for $k\ne0$, 
\begin{align*}
\max_{0\leq j\leq \ell^A-M} d(T^{-j}x,T^{k-j}y) 
&\geq      d(x,T^{k}y)  \\
&\geq d(Z,T^kZ)-d(Z,x)-d(T^ky,T^kZ) \\ 
& > r_0-r_0/16 -r_0/16 = (7/8)r_0 > \eta 
 \ .
\end{align*} 
This proves Claim 1. 
\end{proof} 
\bigbreak

We fix $M,\ell,\eta$ according to Claim 1. Recall $\zeta>0$.

\medbreak\noindent{\bf Claim 2.} {\it  
There is $M_0\in\NN$ such that 
for all  large 
$M\in p\NN$, 
there is 
a family $\Gamma_N$ of $N$-loops based at vertex $0$ 
such that $\# \Gamma_N \geq e^{h'N}$ and the  following holds. 

If $\{\bar x^{\gamma} : \gamma \in
\Gamma_N\}\subset X $ is such that 
$\bar x^{\gamma} |_0^N = \gamma $, for all $\gamma$ in $\Gamma_N$,  
then for all $\gamma\in\Gamma_n$ and $\gamma\ne 
  \gamma'$ in $\Gamma_n$,  
two separation properties are satisfied: 
 \begin{enumerate}
  \item[(S1)] $\pi(\bar x^{\gamma})$ and $\pi(\bar x^{\gamma'})$
are $(\eta,M+M_0,N-(M+M_0))$-separated;
  \item[(S2)] $\pi(\bar x^{\gamma})$
is $(\eta,M+M_0,N-(M+M_0))$-separated from $\pi (\hat z)$ 
whenever $\hat z\in X$, $k\in\ZZ$ and $m:=\lceil\zeta N\rceil$
satisfy (i)$\ \hat z|_k^{k+m}=\ell^\infty|_0^m$  and 
$(ii) \ \ [k,k+m]\subset [M+M_0,N-(M+M_0)] $  . 
\end{enumerate}
}
\medbreak

\begin{proof}[Proof of Claim 2]
We choose 
 $M_0\in p\NN$ such that, for any vertex $v$ 
in the graph $G$, from which there is a path to $0$ of length 
a multiple of $p$, we may fix paths of length $M_0$: $p^{\to v}$ from vertex $0$ to $v$
and  a path  $p^{v\to}$ from $v$ to $0$.

Because $\eta < \eta_1$ and the 
inequality in \eqref{eta_1} is strict,  
there is an $\epsilon >0$ such that 
for any sufficiently large $n$ there is   a 
$(4\eta,n)$-separated subset $S_n$ of 
$\pi([0])$ such that  $\#S_n\geq  e^{(1+\epsilon) h'n}$. 
For each $x\in S_n$, pick $\hat x\in\pi^{-1}(x)\cap[0]$ 
and define the following concatenation:
 $$
   \gamma(\hat x)\  := \ p^{\to\hat x_{-M}}\cdot \hat x|_{-M}^{n+M} \cdot
   p^{\hat x_{n+M}\to} \ \ . 
 $$
Given $n$, define 
$N
= n+2M_0+2M  
$; for $x$ in $ S_n$,  
 $\gamma (\hat x)$ is a loop of length $N$ based at $0$. Define 
\begin{align*}
\widehat{\Gamma_N}\ &=\ \{ \gamma (\hat{x}): x\in S_n \} \ , \\ 
\Gamma_N \ &=\ \{ \gamma \in \widehat{\Gamma_N} : 
\gamma \text{ satisfies } (S2) \} \ .
\end{align*} 
We will show that for all sufficiently large $n$, 
Claim 2 holds for this $\Gamma_N$.  

For distinct $w,x\in S_n$, 
there is an integer $0\leq k<n$ 
such that $d(\pi(\sigma^k\hat w),\pi(\sigma^k\hat x))>4\eta$. 
Hence, given any $\bar w,\bar x$ in $X$ such that 
 $\bar w|_0^N=\gamma (\hat w)$ and  
$\bar x|_0^N=\gamma (\hat x)$, 
we have from Claim 1 some $k$ in the interval $[M+M_0, n+M+M_0]
=[M+M_0, N -(M+M_0)]$ such that 
\begin{align*}
&\    d\big(T^k\pi(\bar w),T^k\pi(\bar x)\big) \\
>\ &\ 
d\big(T^{k-M-M_0}\pi(\hat w),T^{k-M-M_0}\pi(\hat x)\big)-2\eta/4
\      \  > \ \eta \ .
\end{align*} 
This shows that $\Gamma_N$ will satisfy the separation property $(S1)$. 

Let $S'_n$ be the set of points $x\in S_n$ 
such that $\gamma (\hat x)$  fails the separation property $(S2)$.
Pick $H$ such that  $h_\topo(Y)<H<h'/(1-\zeta )$. 
By \eqref{eq.BowenFormulaCover} 
we can find a number $C<\infty$ such that
 \begin{equation} \label{ineq1}
    \forall m\geq0\quad r_{\text{span}} (\eta/2,m,\pi(X),T)\leq C e^{Hm} \ .
\end{equation} 
As $Y=\pi(X)$ is compact and $\pi$ uniformly continuous, 
 \begin{equation} \label{ineq2}
   \exists C'<\infty\ \forall m\geq 0\quad
   r_{\text{span}} (\eta/2,m,\pi[\ell^{[m/\ell]}]_X,T) \leq C'\ .
\end{equation} 

Now suppose 
$m:=\lceil\zeta N\rceil$ 
with $[k,k+m]\subset [M+M_0,N-(M+M_0)]$ as in (S2).
It follows from  
\eqref{ineq1} and \eqref{ineq2} 
that the set of all $\pi (\hat z)$ 
such that 
$\hat z|_k^{k+m}=\ell^\infty|_0^m$ 
is contained in
at most
 $C e^{k H}\times C'\times C e^{(N-k-\zeta N)H}=C'C^2 e^{(1-\zeta) HN}$ 
dynamical $(\eta/2,N)$-balls. 
No such set can contain  
two $(\eta, M+M_0, N-(M+M_0))$-separated points. 
Thus, considering the union over 
$k$
we have  
 $
    \#S'_n\leq N C'C^2 e^{(1-\zeta)HN}
 $
and therefore for large $N=n+2(M+M_0)$ 
and for $C''=e^{-2(M+M_0)}$,
\begin{align} \label{ineqexp}
|\Gamma_N| \ &=\  |\widehat{\Gamma_N}| - | S'_n| \\ \notag 
&\geq \ 
C''e^{(1+\epsilon)h'N} - N C'C^2 e^{(1-\zeta)HN} \ 
> \ e^{h'N}
\end{align}  
where the last inequality holds for large $N$
because $(1-\zeta )H < h'$.
This finishes the proof of Claim 2. 
\end{proof} 

As $X$ has period $p$,
we may fix  loops $L_1,L_2$ based at vertex $0$ 
such that $|L_2|=|L_1|+p \in p\NN$. 
We will have markers of the form 
$m_i:=\ell^A\tilde\ell^CL_i$, $i=1,2$, for some integers $A,C$. 
Fix $N$  satisfying Claim 2. To recognize markers, we fix $C$ and then $A$ large enough so that:
 \begin{equation}\label{eq.AClarge}
    |\tilde \ell^C|>\zeta N +2M+M_0 \text{ and }
    |\ell^A|>\tilde\ell^C+
\max_{i=1,2} L_i+\zeta N+2M+M_0.
 \end{equation}

We consider the subshift of finite type $X_K\subset X$ defined as the 
set of paths obtained from concatenations of 
words of the form 
$m_aw_{1}w_{2}\dots w_{K}$. 
where  $K$ is fixed, but large, $a=1,2$ and 
$w_1,w_2,\dots,w_K\in \Gamma_N$. 

Observe that $X_K$ is irreducible and its period is a multiple of $p$
and divides the two lengths $|m_a|+K|w_i|$, for $a=1,2$ (and any
$i$). These lengths differ by $p$, hence the period of $X_k$ 
is equal to $p$. By (\ref{ineqexp}), the topological entropy of
$X_K$ has the bound:
 $$ 
h_\topo(X_K)\ \geq \ \frac{K\log\#
\Gamma_N}
{KN+|m_2|}\ 
>
\ \frac1{1+\frac{|L_2|+|\ell^A\tilde\ell^C|}{KN}}h'\ ,
$$
with the right side greater than $h$ for large $K$ (given $N$). 
It only remains to show that $\pi : X_K\to Y$ is
 injective.
 Let $\bar x,\bar y\in X_K$ with $\pi(\bar x)=\pi(\bar y)$.
\smallbreak

We first prove $M(\bar x)=M(\bar y)$ where $M(\bar x)$ is the set of positions where a marker $m_i$ appears.
Assume that $0\in M(\bar x)$ so: $\bar x|_0^{\ell^A}=\ell^A$. We claim that the corresponding subword of $\bar y$ must also be part of marker (mostly). Indeed, the
separation property (S2) from Claim 2 implies that, if $\bar
y|_n^{n+N}=w_i$ 
in $\Gamma_N$, then $[n+M+M_0,n+N-M-M_0]$ cannot overlap
$[0,\ell^A]$ on a set of length $\geq\zeta N$. Thus, $\bar y|_{\zeta
  N+M+M_0}^{\ell
^A-\zeta N-M-M_0}$ occurs in $\bar y$ as part of
a marker $m_i=\ell^A\tilde\ell^CL_i$ ($i=1$ or $2$).

It follows that $M(\bar y)$ contains some $k$ with
$-\tilde\ell^CL_i-\zeta N-M-M_0\leq k\leq \zeta N+M+M_0$. Thanks to
\eqref{eq.AClarge}, $-\ell^A+M\leq k\leq \tilde\ell^C-M$ and Claim 1
applied to $\sigma^{\ell^A}\bar x,\sigma^{\ell^A-k}\bar
y\in[\ell^A.\tilde\ell^C]$ 
yields $k=0$. It follows that $\bar M(x)= M(\bar y)$ by symmetry.

\smallbreak

Let $n_1<n_2$ be two consecutive elements of $M(\bar x)=M(\bar y)$. Considering $\bar x$ and $\bar y$, we have: $n_2-n_1=|m_i|+KN=|m_j|+KN$. Thus $|m_i|=|m_j|$, so $m_i=m_j$ as the lengths are pairwise distinct. Let $r:=n_1+|m_i|+sN<n_2$ for some positive integer $s$. Observe $\bar x|_r^{r+N}=w_i$, $\bar y|_r^{r+N}=w_j$ for some $i,j\in I_N$. If $i\ne j$, then,  $0=d(\pi(\sigma^{r+k}\bar x),\pi(\sigma^{r+k}\bar y))>d(\pi(\sigma^kx_i),\pi(\sigma^kx_j))-2\eta$ but this should be positive for some $k\in[M+M_0,N-M-M_0]$. Thus $i=j$. As $\inf M(\bar x)=-\infty$ and $\sup M(\bar x)=\infty$, $\bar x=\bar y$, concluding the proof.
\end{proof}


Theorem \ref{thm.continuousfactorMS} is now an obvious consequence of the next Proposition (whose proof follows).

\begin{proposition}\label{p.hstar}
Let $\pi:(X,S)\to(Y,T)$ be a continuous factor map from an irreducible, period $p$ Markov shift into a self-homeomorphism of a Polish space. For any $\mu\in\Proberg(S)$ and $h<h(T,\pi_*\mu)$, there exists $\nu\in\Proberg(S)$ with compact support and
 $
    h(T,\pi_*\nu) > h.
 $
In particular,
 \begin{equation}\label{e.hstar}\begin{aligned}
   \sup \{h(T,\pi(\Sigma)):\Sigma\subset X,\; \Sigma &\text{ irreducible period $p$ SFT}\}\\
    & = \sup \{ h(T,\pi_*\mu):\mu\in\Proberg(S)\}.
 \end{aligned}\end{equation}
\end{proposition}


To prove the above proposition, we need some
definitions and notations.
For a Borel partition $P$, $\partial P$ denotes the union of the
boundaries of the elements of $P$. For $x\in X$, $P(x)$ is the unique element of $P$ containing $x$. $P^n$ is the set of words  $v=v_0\dots v_{n-1}$ on $P$ of length $n$. Any such word defines a cylinder $[v]:=v_0\cap T^{-1}v_1\cap\dots\cap T^{-n+1}v_{n-1}$. $v$ is the $P,n$-name of any point in $[v]$. $P^n$ will also denote the set of cylinders defined by words on $P$ of length $n$. Depending on the setting $P^n(x)$ will mean either the $P,n$-name or  cylinder of $x$. 

\begin{proof}[Proof of \Prop\ref{p.hstar}]
Let $\delta h:=h(T,\pi_*\mu)-h>0$. 
As $Y$ is Polish, there exists a finite Borel partition $P$ such that
 \begin{equation}\label{e.deltaP}
    h(T,\pi_*\mu)<h(T,\pi_*\mu,P)+\delta h/10 \text{ and }\pi_*\mu(\partial P)=0.
\end{equation}
Fix $t_0>0$ such that, for all large $n$,  the number of subsets of $\{1,\dots,n\}$ with cardinality at most $t_0n$ is less than $e^{(\delta h/20)n}$.
 As $\pi$ is continuous, there exist an integer $M$ and a Borel set $X_1\subset X$ such that $\mu(X_1)>1-\min(\delta h/(40\log\#P),t_0/2)$ and
 $$
   \forall x\in X_1\;\forall w\in X\; x|_{-M}^M = w|_{-M}^M \implies P(\pi(x))=P(\pi(w)).
 $$
Let $0$ be a vertex of $G$ with $\mu([0])>0$. 
Define $X_0$ to be the set of points in $X$ such that $x_n=0$ for 
infinitely many positive $n$ and also for infinitely many  negative
$n$. By ergodicity, 
$\mu(X_0)=1$.
\begin{claim}\label{periodclaim}
There exists a period $p$ SFT 
$\bar X\subset X_0$ and a continuous factor map 
$p:X_0\to\bar X$ such that, if 
$X_2:=\{x\in X_0:p(x)|_0\ne x|_0\}$, then:
 \begin{equation}\label{e.nonunifCon}
   \mu(X_2)<\frac{\min\left(\delta
       h/(40\log\#P),t_0/2\right)}{2M+1} \ . 
 \end{equation}
\end{claim}

\begin{proof}[Proof of Claim \ref{periodclaim}]
The loop graph at $0$ is the
graph $\hat G$ with vertices: $0$
and $(w,k)$ if $0<k<n$ and $w\cdot 0$ is a word of $X$ of length $n+1$ with $w_0=0$
and $0\notin\{w_1,\dots,w_{n-1}\}$; edges: $0\to(w,1)$,
$(w,k)\to(w,k+1)$ for $0<k<n-1$ and $(w,n-1)\to 0$. 
The \emph{loop shift} (see, e.g., \cite{BBG2006}) for $G$ at $0$ 
is  the Markov shift  $\hat X$ presented 
by $\hat G$. Note, $\hat X$ like $X$ has period $p$. 
Let $\psi:X_0\to \hat X$ be the obvious topological
conjugacy.

%

Given an enumeration $w^1,w^2,\dots$, of the words of $X$, let 
$\hat X_N$
be the SFT defined by the finite subgraph $\hat G_N$ of $\hat G$
obtained by restricting the previous construction to the words $w^n$
for $n\leq N$. 
We fix $N$ large enough so that $\hat G_N$ has the same
period $p$ (g.c.d. of loop lengths) as $\hat G$; for all
  $n\geq N$, $np$ is a sum of lengths of first return loops to $0$ in 
$\hat G_N$;
 and $[0]\cup\bigcup\{[(w^n,k)]_{ X}: n\leq N, 0<k<|w^n|\}$ has
 $\psi_*\mu$-measure close enough to $1$ that \eqref{e.nonunifCon}
   will hold.
Then we define the SFT $\bar X=\psi^{-1} \hat X_N\subset X_0$.

We can define a map 
$q: \hat X\to \hat X_N$ by replacing each $w^n$, $n>N$, by some
  concatenation $\tilde w^n$ of $w^i$'s for $i\leq N$ with total
  length $|w^n|$ (making choices depending only on $|w^n|$). 
We define $p:X_0\to\bar X$ by $p=\psi^{-1}\circ q\circ \psi$. 
\end{proof}

We denote by $\bar\pi$ the restriction of $\pi$ to $\bar X\subset X$ and set $\nu:=p_*\mu$.

Observe that, for $x\in X$:
 $$\begin{aligned}
   &P(\bar\pi p(x))\ne P(\pi(x))\implies
      x\notin X_1 \text{ or } p(x)|_{-M}^{M} \ne x|_{-M}^{M}\\
   &p(x)|_{-M}^{M} \ne x|_{-M}^{M}\implies 
     x\in S^{-M}X_2\cup\dots\cup S^MX_2.
 \end{aligned}$$
Hence, by the Birkhoff ergodic theorem, there exists $X_3\subset X$ such that $\mu(X_3)>9/10$ and for all large $n$, all $x\in X_3$, 
  \begin{equation}\label{e.cylAm}
    \frac1n\#\{0\leq k<n:P(\bar\pi p(T^kx))\ne P(\pi(T^kx))\} < \rho:=\min\left(\frac{\delta h}{20\log\#P},t_0\right).
  \end{equation}

For any two words $v,w\in P^n$, define the relation:
 $$
    v\sim w\iff \#\{0\leq k<n:v_k\ne w_k\}<\rho n.
 $$
Note that for $v\in P^n$ for $n$ large enough, 
by choice of $t_0$ we have
 \begin{equation}\label{e.simP}
   \#\{w:w\sim v\} \leq  e^{(\delta h/20)n}\times \#P^{\rho n}\leq  e^{(\delta h/10)n}\ .
 \end{equation}

The theorem of Shannon-McMillan-Breiman applied to $(T,\bar\pi_*\nu)$ gives sets $E_n$ of $P,n$-words  such that, for all large $n$, writing $[E_n]:=\bigcup_{v\in E_n} [v]$,
 \begin{equation}\label{e.SMB1}
   \bar\pi_*\nu\left([E_n]\right)>9/10\text{ and }\#E_n\leq\exp (h(T,\bar\pi_{*}\nu)+\delta h/10)n.
 \end{equation}
Let $F_n:=p^{-1}\bar\pi^{-1}([E_n])\cap X_3$. It is a Borel set. $\pi(F_n)$ is Borel (up to a subset included in a set with zero $\pi_*\mu$-measure). Using $\bar\pi_*\nu:=\mu\circ p^{-1}\circ\bar\pi^{-1}$,
 $$
   \pi_*\mu(\pi(F_n))=\mu(\pi^{-1}\pi(F_n))\geq\mu(F_n)\geq  \bar\pi_*\nu([E_n])- \mu(X\setminus X_3)>8/10.
 $$
Let $n$ be large and $x\in F_n\subset X_3$. Eq. \eqref{e.cylAm} gives:
 $$
     P^n(\pi(x))\subset \bigcup\bigl\{[v]:v\sim P^n(\bar\pi p(x))\bigr\}.
 $$
By construction of $F_n$, $\bar\pi p(x)\in[v]$ for some $v\in
E_n$. Thus, using eqs. \eqref{e.SMB1} and \eqref{e.simP},
$G_n:=\bigcup_{v\in E_n} \{w:w\sim v\}$ satisfies 
$[G_n]\supset\pi([F_n])$ and therefore:
 $$\begin{aligned}
    &\pi_*\mu([G_n])\geq  \pi_*\mu(\pi(F_n)) > 8/10 \text{ and } \\
    &\# G_n\leq \# E_n \times \exp (\delta h n/10)
      \leq \exp \left((h(T,\bar\pi_*\nu)+\tfrac{2}{10}\delta h)n\right).
 \end{aligned}$$
Applying the Shannon-McMillan-Breiman Theorem this time to $\pi_*\mu$ and $P$ and recalling \eqref{e.deltaP}, we get:
 $$
     (1+\delta)h=h(T,\pi_*\mu) \leq h(T,\pi_*\mu,P)+\delta h/10 \leq h(T,\bar\pi_*\nu) + \tfrac25\delta h.
 $$
Hence, $h(T,\bar\pi_*\nu) >h$, proving the first claim of the Proposition.

Observe that the supremum over measures in eq. \eqref{e.hstar} is at
least equal to that over SFTs $\bar X\subset X$: apply the variational
principle to each compact t.d.s. $(T,\pi(\bar X))$. Conversely, given
$\mu\in\Proberg(S)$ and $h<h(T,\pi_*\mu)$, the first claim of the
Proposition gives an SFT $\bar X\subset X$ carrying an ergodic measure
$\nu$ with $h(T,\pi_*\nu)>h$. But $h(T,\pi_*\nu)\leq h(T,\pi(\bar
X))$, so $h(T,\pi(\bar X))>h$.  By enlarging the SFT $\bar X\subset
X$, one can reduce its period to that of $X$. 
The equality of the suprema in eq. \eqref{e.hstar} is now obvious.
\end{proof}

\section{Bowen  factors of Markov shifts} 
\label{sec.bowen}
In this section we prove 
Theorem \ref{bowentypefactors},
which states
  conditions satisfied by Sarig's symbolic dynamics under which a
  factor of a Markov shift is almost-Borel isomorphic to a Markov
  shift.

Recall  
Definition \ref{def.AlmostBowen}
for  Bowen type factor maps.
For a factor map $\pi$ which is Bowen type on 
its domain, the set of relations $\sim$ 
satisfying (1) and (2)
in Definition \ref{def.AlmostBowen}, if it is nonempty, contains a minimal
relation, for which two symbols are related if and only if the 
images of their time zero cylinder sets have nonempty 
intersection.
%
%
%
 A prototypical Bowen type map is a one-block 
code from an SFT onto a sofic shift; in this case,  
the relation $\sim$ on symbols 
is transitive.  When the factor system $Y$ is not zero 
dimensional, the relation $\sim$ on symbols cannot be 
transitive. 
For our almost-Borel purposes,
the condition (2) 
in Definition \ref{def.AlmostBowen} is only a notational 
convenience.


\begin{definition} \label{defnxret}
Let $X$ be a 
Markov shift with alphabet $\mathcal A$. 
For $a, b\in \mathcal A$, 
$X_{a,b}$ is the set of $x$ in $X$
such that $x_n =a$  for infinitely many 
negative $n$ and $x_n=b$ for infinitely 
many positive $n$. $X_a$ 
is the subset of $X$ consisting of points 
$x$ such that $x_n=a$ for infinitely 
many positive $n$ and infinitely many 
negative $n$. 
The {\it return set} of $X$ is 
$X_{\text{ret}}:= 
\cup_a X_a$.  
The {\it\sarig\ set} of $X$ is 
$X_{\pm \text{ret}}:= 
\cup_{a,b} X_{a,b}$.
\end{definition}

One virtue of the \sarig\ set 
of a Markov shift $X$ is that it contains every 
compact subshift of $X$.

\medbreak

We will use the following consequence of Theorem \ref{thm.continuousfactorMS}
to establish the universality 
claim of Theorem \ref{bowentypefactors}.

\begin{proposition}\label{bowentypeuniv} 
Let $\pi: (X,S)\to (Y,T) $ be a Borel factor map,  
from an irreducible Markov shift of period $p$.  
Assume that it is countable 
to one and Bowen type on 
the \sarig\ set $X_{\pm\ret}$. Then 
$(\pi(X_{\pm\ret}),T)$ is $(h(S),p)$-universal. 
\end{proposition}

The Bowen type assumption is key here - compare with Rem. \ref{badborel}.

\begin{proof}  
It suffices to show that $\pi(X_{\pm\ret})$ is $(h(S)-\epsilon,p )$ 
universal for every $\epsilon > 0$ (Prop.\ \ref{p.UnivMark}). 
Given $\epsilon$, let 
$\Sigma$ be an irreducible SFT 
of period $p$ contained in $X_{\pm\text{ret}}$ 
such that $h(\Sigma )>h(S)-\epsilon$.
Let $\overline{\Sigma}$ be $\pi \Sigma$, endowed with the 
quotient topology; 
as in \cite{Friedfp} $\overline{\Sigma}$  is a compact metrizable dynamical system -- use,
e.g., \Prop\ref{compactrelation} with $\Sigma$  compact metrizable and the quotient relation a
closed set in $\Sigma\times\Sigma$ ($\pi$ is Bowen type on $\Sigma\subset X_{\pm\ret}$).
It follows from  \Thm\ref{prop.continuousfactorSFT} 
that  
$\overline{\Sigma}$ is  $(h(S) -\epsilon,p)$-universal. 

A countable-to-one map from a standard Borel space into another one 
has a Borel section
\cite[(18.10) and (18.14)]{Kechris}.
 It follows that 
$\pi \Sigma$ is a Borel set, and a set is Borel 
in $\overline{\Sigma}$ or $\pi\Sigma$  if and only if 
its preimage in $\Sigma$ is Borel.  
Consequently 
the 
identity $\overline{\Sigma} \to \pi \Sigma$ 
is a Borel isomorphism.  Therefore 
$\pi \Sigma$, like $\bar\Sigma$, is
 $(h(S) -\epsilon,p )$-universal. 
\end{proof}

The key step 
for the proof of Theorem \ref{bowentypefactors}
is the following. 
We will let $\mathcal A(S)$ or $\mathcal A(X)$ 
denote the alphabet (symbol set) of a shift space $(X,S)$. 
In the setting of Theorem \ref{bowentypefactors}, we have:

\begin{proposition}\label{prop.MSforMeasure}
Let $(X,S)$  be a Markov shift and let
$\pi: (X,S)\to (Y,T)$ and $\bar X$ be as in \Thm\ref{bowentypefactors}: $X$ satisfies \eqref{eq.fec0}, $\pi$ is a Borel 
factor map such that 
for each irreducible component $Z$ of $X$,
\begin{enumerate}   
\item 
$\pi$ is Bowen type on the \sarig\ set $Z_{\pm\text{ret}}$, and 
\item 
the restriction $\pi|Z_{\pm ret}$ is finite-to-one;
\end{enumerate} 
and, $\bar X$ is the union of the \sarig\ sets $Z_{\pm\ret}$ of the
irreducible components $Z$ of $X$.

Then 
the induced map 
$\pen (\bar X) \to \pen (\pi \bar X)$
is surjective.  
Moreover, there is a countable collection of Borel
 factor maps 
$\pi': (X',S') \to (Y',T')\subset (Y,T)$
for which the following hold.  
\begin{enumerate} 
\item 
$(X',S') $ is an irreducible  Markov shift.
\item 
 $\pi'$ is both
Bowen type and finite to one on the \sarig\ set $X'_{\pm\ret}$.
\item 
If $\nu \in \pen (T|\pi(\bar X))$, then there exists 
some $\pi'$ in the collection and some
$\mu' \in \pen(S') $ such that 
$\pi': (S',\mu')\to (T,\nu ) $ is a 
measure-preserving isomorphism.
\end{enumerate} 
\end{proposition}

\begin{remark}Even though measures are supported on the 
return sets, our proof of Proposition \ref{prop.MSforMeasure}
appeals to 
$\pi$ being Bowen type on the (larger) \sarig\ sets.
\end{remark}

\begin{proof}[Proof of Proposition \ref{prop.MSforMeasure}]
Let $\nu \in \pen (T|\pi\bar X)$. 
The set $\pi\bar X$ is the union of the countable collection of 
invariant sets $\pi Z_{\pm \text{ret}}$. 
Since $\pi|\bar X$ is at most countable to one, these sets
are Borel. As $\nu$ is ergodic, 
there exists $Z$ such that 
$\nu (\pi Z_{\pm \text{ret}})=1$. 
Because $\pi $  is finite to one 
on $Z_{\pm\ret}$ 
there exists $\mu\in\pen(S|Z_{\pm\ret})$ with $\pi \mu = \nu$
(\Prop\ref{finitetoonelift} and ergodic decompositon). 

Thus $\pen(\bar X)\to\pen(\pi\bar X)$ is surjective, as claimed. The rest of
the proof is devoted to the construction of the factors maps $\pi':(X',S')\to(Y',T')$.

Because 
$\mu $ is ergodic, there is a positive integer $m$ and 
a set $E$ in $Z_{\pm\ret}$ of $\mu$-measure one 
such that for every $y$ in $\pi E$:
\begin{itemize} 
\item 
$y$ has exactly $m$ preimages in $E$, and 
\item 
with $\nu_y$ denoting the measure assigning mass 
$1/m$ to each preimage point of $y$ in $E$, 
for every Borel set $B$ in $X$ 
\[
\mu B = \int_Y \nu_y \big( (\pi^{-1}y)\cap B\big)\, d\nu (y) \ . 
\]
\end{itemize}  
If $m=1$, $\pi':=\pi$ already satisfies condition (3). Now suppose $m>1$. 
Let $E_k$ be the set of $x$ in $E$ such that, if 
$x^1, \dots ,x^m$ are the distinct preimages in $E$ of $\pi x$, 
then the $m$ words $x^i[-k,k]$ are distinct. 
For large enough $k$, $\mu E_k >0$. After passing to a 
higher block presentation of $(Z,S)$, we may assume $k=0$. 

Let $\sim$ be some relation on $\mathcal A  (Z)$ 
with respect to which $\pi $ is Bowen type on $Z_{\pm\ret}$. 
Let $(F_m,S_m) $ denote the $m$-fold fibered product system 
of $(Z,S|_Z)$ over $\sim $.  
Here 
\[
F_m : =\{x=(x^1, \dots , x^m ) \in Z^m: x^i \sim x^j, 
1\leq i\leq j \leq m \} 
\]  
(recall $x^i\sim x^j$ means $x^i_n\sim x^j_n$ for all $n$)  
and $S_m$ is the restriction to $F_m$ of the product 
map $S\times \dots \times S$. 
Thanks to the Bowen property, $(F_m,S_m)$ is a Markov shift, whose 
alphabet $\mathcal A(S_m)$ is
a subset of the set of $m$-tuples of symbols from 
$\mathcal A(Z)$  which are mutually related.
For $1\leq r \leq m$, 
let $p_r:F_m\to X$ be the coordinate projection 
map $x\mapsto x^r$. 
Define $\widetilde{\pi}: F_m \to Y$ as the composition 
$\widetilde{\pi}=\pi\circ p_r$, for any $p_r$. Here 
$\widetilde{\pi}$ is 
well defined since $x\sim y\implies\pi(x)=\pi(y)$, for all $x,y\in Z$.

 We define an 
$S_m$-invariant measure $\widetilde {\mu} $ on $F_m$ 
as follows. For each $y $ in $\pi E$, 
define a measure $\widetilde{\nu}_y$ on 
the $\widetilde{\pi}$-preimages 
 $y$ as follows: 
$\widetilde{\nu}_y $ 
assigns mass $1/m!$ to 
each $m$-tuple 
$(x^1, \dots , x^m)$ such that  the $m$ entries are 
distinct preimages of $y$ (there are $m!$ such tuples for $\mu$-a.e.\ $y$).
Then for any Borel set $B$ in $F_m$ define 
\[
\widetilde{\mu} B = 
\int_Y 
\widetilde{\nu}_y \big( (\pi^{-1}y)\cap B\big)\, d\nu (y) \ . 
\]
Then $p_r\widetilde{\mu}= \mu$ and 
$\widetilde{\pi}\widetilde{\mu}=\pi\mu=\nu$ . 
Because $\mu$ is ergodic, we may take $\mu''$ an 
ergodic measure from the ergodic decomposition of 
$\widetilde{\mu}$ such that 
$p_r\mu''=\mu$, for $1\leq r\leq m$, and 
$\widetilde{\pi}\mu''=\nu$. 

\begin{claim} \label{claim}
For $\mu''$-a.e.\ $x=(x^1,\dots,x^m)\in  F_m$, for all $n\in\ZZ$: \\ 
(i) the $m$ symbols $x^1_n,\dots,x^m_n$ are pairwise distinct; \\ 
(ii) for $1\leq i,j\leq m$:  
 $x^i_nx^j_{n+1}$ is an $S$-word if and only if
 $j=i$.
\end{claim}

\begin{proof}[Proof of Claim \ref{claim}] 
Because $p_1\mu''=\mu$ and $\mu E_0>0$,
the set 
\[
E''_0:=
\{(x^1,\dots ,x^m)\in F_m: 
x^i_0 \neq x^j_0, 1\leq i< j \leq m\} 
\]
satisfies $\mu''E_0''=\mu E_0>0$. Let $a=(a^1, \dots , a^m)$ be an 
$m$-tuple of distinct symbols 
such that  $[a]:=\{x\in F_m:x_0=a\}
\subset E_0''$ satisfies $\mu''[a]>0$.

We note that (i) follows from (ii) and prove this last assertion of
the claim. 
For a contradiction, assume that there are symbols 
$b=(b^1,\dots ,b^m)$ and $c=(c^1,\dots ,c^m)$ in $\mathcal A(S_m)$ such 
that $\mu''[bc]>0$ 
 and (say) $b^2c^1$ is an 
$S$-word
(i.e. the transition $b^2 \to c^1$ is allowed in $S$).

The following hold
for all $x=(x^1,\dots ,x^m)$ from a set of 
full $\mu''$ measure, (1) because, for each $r$, $p_r(\mu'')=\mu$ which is ergodic and (2) by ergodicity of $\mu''$:
\begin{enumerate} 
\item 
There is a symbol
which in every $x^i$ occurs with positive frequency in positive 
and in negative coordinates.
\item  
There are sequences of 
integers $(i_n)$, $(j_n)$ (depending on $x$) with $i_1<j_1<i_2<j_2< \cdots $ 
such that for all $n$, $x_{j_n}x_{j_{n}+1}=bc$ and $x_{i_n}=a$. 
\end{enumerate} 
Pick one such $x$. For
each $n\geq 1$, define a point $z^{(n)}$ in $S$ by setting 
\begin{align*} 
(z^{(n)})_t\  & =  \ (x^1)_t \quad \text{if } t\geq j_n + 1 \\ 
        & = \ (x^2)_t \quad \text{if } t\leq j_n  \ .
\end{align*} 
Then for all $n$, $z^{(n)}\sim x^1$, so $\pi (z^{(n)})=\pi(x^1)$.
If $\ell > n$, then $(z^{(n)})_{i_{\ell}}=a_1$ and 
$(z^{(\ell)})_{i_{\ell}}=a_2$, so $z^{(n)}\neq z^{(\ell)} $.  
By condition (1), the points $z^{(n)}$ are all in $Z_{\ret}$.  
This contradicts $\pi$ being finite to one on 
$Z_{\pm\ret}$, and   
 proves (ii).
 \end{proof}

\medbreak

Let $(\widetilde{X_m},\widetilde{S_m})$ be the Markov shift 
contained in the Markov shift 
$(F_m,S_m)$ and which is defined by the following
conditions:
\begin{enumerate} 
\item 
$\mathcal A(\widetilde{S_m})$ is the set of  
$a=(a_1,\dots , a_m)$ in 
$\mathcal A(S_m)$ such that the symbols $a_1, \dots ,a_m$ from 
$\mathcal A(S) $ are distinct.
\item 
There is a transition from $
 a=(a_1,\dots , a_m)$ to 
$
b=(b_1,\dots , b_m)$ if and only if 
the following holds: 
for $1\leq i,j\leq m$ 
there is an $S$ transition $a_i \to b_j$ if and 
only if $i=j$. 
\end{enumerate} 
The claim \ref{claim}
implies that $\mu''$ assigns measure 
one to the Markov shift $\widetilde{X_m}$. 
By ergodicity of $\mu''$, there is a unique 
irreducible component $(X'',S'')$
of $\widetilde{X_m}$ such that $\mu'' X'' =1$. 

Now define $(X',S')$ to be the shift space (on 
a countable alphabet) 
which is the image of 
$(X'',S'')$ under the one-block map $\psi$ 
defined by the rule
$\psi: (a_1, \dots , a_m)\mapsto \{a_1, \dots , a_m\}$. 
The map $\psi$ is right resolving: i.e., 
if $A_0A_1$ is a word of length two occuring in a 
point of $X'$, and $\psi: \widetilde{a_0}\mapsto A_0$, then there 
exists a unique symbol $\widetilde{a_1}$ following $\widetilde{a_0}$ in
$\widetilde{X_m}$ such that $\psi: \widetilde{a_1}\mapsto A_1$. 
The map $\psi$ is likewise 
left resolving.  Therefore $X'$ is a Markov shift and it is also
irreducible. Thus, for every $x$ in $X'$, for every $\widetilde a$ in 
$\mathcal A(\widetilde{X_m})$ such that $\psi: \widetilde a\mapsto x_0$, 
there exists a unique preimage $\widetilde x$ of $x$ such that $\widetilde{x_0}
=\widetilde a$. Every point of $X'$ has exactly $m!$ preimage points 
in $\widetilde{X_m}$.

The map $\psi$ only collapses points which have the same image 
under $\widetilde{\pi}$. Therefore there is a Borel map
$\pi':(X',S')\to (Y',T')$ 
defined by $\widetilde{\pi}= \pi' \psi$, 
where $Y'= \tilde\pi(X'')=\pi'(X')$ 
and $T'$ is the restriction of $T$ to $Y'$. 
Let $\sim$ also denote the natural relation 
on the alphabet of $X'$: 
$\{a_1, \dots , a_m\} \sim \{b_1, \dots , b_m\}$ 
iff $a_i\sim b_j$ for all $i,j$.  
If $w',x'$ are in $X'_{\pm\ret}$, there are 
$w'', x''$ in $  X''_{\pm\ret}$ such that $\psi x'' = x'$, 
$\psi w''=w'$. 
(This is the one point where the proof would fail if we 
used $Z_{\ret}$ rather than $Z_{\pm\ret}$.)
Then $p_1 x'' = x \in Z_{\pm\ret}$
 and 
$p_1 w'' = w \in Z_{\pm\ret}$; and, 
$\pi' (x')=\pi'(w')$ if and only if 
$\pi (x)=\pi(w)$. Because $\pi$ is Bowen type 
on $Z_{\pm\ret}$, 
it follows that 
$\pi'$ is Bowen type 
on $X'_{\pm\ret}$. 

A set in $X''$ of full measure for $\mu''$ is 
$E''=\{(x_1, \dots , x_m)\in X'': x_i \in E, 1\leq i \leq m\}$. 
Points in $E''$ with the same $\widetilde{\pi}$ image are mapped by $\psi$
to the same point in $X'$. Setting $\mu' =\pi\mu''$,  the map 
\[
\pi': (S',\mu') \to (T,\nu ) 
\] 
is  an isomorphism of measure-preserving systems.  

The Markov shift $(X'',S'')$ constructed above given $\nu$ was  an irreducible component of the Markov shift obtained by restricting $\widetilde{X_m}$ to a higher block presentation. The higher block presentation was a notational convenience, but in any case there are only countably many higher block presentations of a given $\widetilde{X_m}$. 
Any Markov shift has only countably  many irreducible
components. Consequently, we build only countably many irreducible Markov shift extensions. 
\end{proof}

\begin{proof}[Proof of Theorem \ref{bowentypefactors}]
\Prop\ref{prop.MSforMeasure} implies \jb{the surjectivity} of the induced map 
$\pen (S|\bar X) \to \pen (T|\pi \bar X)$.
\jb{The  characterization of Markov shifts in terms of universal subsystems (\Thm\ref{thm.MarkovChar}) will yield the almost-Borel isomorphism of $\pi(\bar X)$ to a Markov shift as follows.}

Let $\nu$ be an ergodic and invariant probability measure of
$(\pi(\bar X),T)$. Let $\pi':(X',S')\to(Y',T)$ be the  extension
given by \Prop\ref{prop.MSforMeasure} 
with $\mu'\in\pen(S')$ such that $\pi'\mu'=\nu$.
Letting $q$ denote the period of the irreducible Markov shift 
$(X',S')$, we note:
\begin{enumerate}
\item 
The set of periods of $(T,\nu)$ coincides with that of $(S',\mu')$ and
therefore 
contains $q$; 
\item 
The image of $(X'_{\pm\ret},S')$ contains a strictly 
$(h(S'),q)$-universal system (by Proposition \ref{bowentypeuniv}, 
because $\pi'$ is 
finite to one, Bowen type on $X'_{\pm\ret}$). 
\end{enumerate} 
Using that entropy is a Borel function of the measure and the Borel Periodic Decomposition  
(\Thm\ref{thm:specdec}), we obtain an invariant Borel subset
$Z\subset\pi'(X')$ such that, for all measures $m$ on
 $\pi'(X'_{\pm\ret})$, $m(Z)=1$ if and only if
$q$ is a period of $m$ and $h(T,m)<h(S')$.
It follows from (2) above that $Z$ is
strictly $(h(S'),q)$-universal. Note that $Z$ depends only on the extension $\pi'$, hence there are at most countably many such sets $Z$, also: $u_{T}(q)\geq h(Z)=h(S')$.

Thus, either $\mu'$ is the measure of maximal entropy for 
$(X',S')$, or $h(T,m)=h(S',\mu')<h(S')$ so $m(Z)=1$. 
Altogether, then, $(\pi(\bar X),T)$ 
is almost-Borel isomorphic to a countable union of:
\begin{enumerate}
\item strictly $(u_{T}(p),p)$-universal systems (using Lemma \ref{lem.unionSUp});
\item systems supporting a single
measure $\mu$ of $\pen(T)$, 
such that  there exists $p$ 
with $h(T,\mu ) = u_T(p) $ and $(T,\mu )$ 
is $p$-Bernoulli. 
\end{enumerate}
\Thm\ref{univpartexists} implies that $\pi(\bar X)\setminus \pi(\bar X)_U$ (in the notation of that theorem) carries only measures from (2) above.
By \Thm\ref{thm.MarkovChar},
it follows  that $\pi(\bar X)$ is 
almost-Borel isomorphic  to a Markov shift. 
\end{proof}


\section{Continuous factors of Markov shifts: pathology} 
\label{pathology} 
The results of this section will give limits 
to any strengthening of our two main theorems 
(\ref{bowentypefactors} and \ref{thm.continuousfactorMS}) 
about continuous factors of Markov shifts. Recalling the discussion
after Theorem \ref{thm.continuousfactorMS} we build  examples with 
large sets of
\begin{itemize}
 \item[-] measures with entropy greater than the entropy $h_*(\pi)$ from Theorem  \ref{thm.continuousfactorMS} in \Prop\ref{countablebad}.
 \item[-] \mme's for a factor which is not finite to one, in Corollary \ref{sfttopcor}.
 \item[-] period-maximal measures for a finite-to-one but not Bowen
   type factor in 
Corollary \ref{badfinite}.
 \end{itemize}
We also remark that a factor of an irreducible Markov 
shift by a continuous map need not be a factor by a  
Bowen type map, even if it is a compact expansive system. 
Indeed, among subshifts  
(up to topological conjugacy, the compact zero-dimensional expansive
  systems), the  
continuous factors of 
irreducible Markov shifts are exactly the 
coded systems \cite{Fiebigs2002}. 
But among these, the 
factors by one-block codes
are the 
factors by Bowen type maps, and 
form a proper subset of the coded systems 
\cite{Fiebigs2002}.

\subsection{Arbitrary dynamics in high entropy}
It is well known that the entropy of irreducible 
Markov shifts can increase under one-block codes (which are 
continuous and Bowen type factor maps);   
see e.g. \cite{Fiebig2002,Fiebigs2002,FiebigRoy2006,Petersen1986}. 
The following construction, resembling 
\cite[Examples 3.3,3.4]{Petersen1986}, further
shows that a one-block code image of the 
nonrecurrent part
of a Markov shift  can have virtually no 
almost-Borel relation to that Markov shift. 
The quantity $h_*(\pi)$ in the statement of 
Proposition \ref{countablebad} 
comes from Theorem 
\ref{thm.continuousfactorMS}.

\begin{proposition} \label{countablebad} 
Suppose $Y$ is a subshift of $\{ 0,1\}^{\ZZ}$ 
and $\epsilon >0$.  
Then there is a locally compact irreducible Markov shift $X$ and 
a one-block code $\pi$ from $X$ into 
$\{0,1,2\}^\ZZ$ such that 
 $X$ is the disjoint union of Borel 
subsystems $X',X'',X'''$ for which the following hold. 
\begin{enumerate} 
\item 
$\pi(X')$ is almost-Borel isomorphic to $X$ with
  $\pi|X'$ one-to-$1$ ;
\item  
$\pi(X'')$ is almost-Borel isomorphic to $Y$ with $\pi|X''$ countable-to-$1$.
\item
$\pi(X''')$ is a fixed point and $X'''$ is a finite orbit.
\item 
$h_*(\pi )=h(X) < \epsilon$. 
\item 
$\pi(X)$ is compact and almost-Borel isomorphic to the disjoint 
union of $Y$ and $X$. 
\end{enumerate} 
\end{proposition}

\begin{proof} 
We build in stages a labeled graph $G$ defining $\pi$. 
The Markov shift $X$ will be the edge shift defined by 
$G$. Each edge  will be labeled by a symbol from $\{0,1,2\}$. 
The one-block code will be the rule replacing an edge 
with its label. 

First, there is a labeled subgraph $G^+$ 
which has for every $Y$-word 
$W$(including the empty word $\nullset$) a vertex $v_W$, 
and 
for $i\in \{0,1\}$ with $Wi$ a $Y$-word, has an 
edge labeled $i$ from $v_W$ to $v_{Wi} $. 
Then for each $z$ in $Y$, there is a unique path 
from $v_{\nullset}$ labeled by the onesided sequence 
$z[0,\infty )=z_0z_1\dots $. 
Similarly build a graph $G^-$
such that for each $y$ in $Y$ there is a unique 
left infinite path into $v_{\nullset}$ labeled by $y(-\infty,-1]= ... y_{-2}y_{-1}$.

Let $X''$ be the edge shift presented by 
$G^- \cup G^+$.  Note, $v_\nullset$ is the only common vertex of $G^-,G^+$.
The image $\pi X''$ is the 
set of all shifts of sequences that are concatenations 
$y(-\infty -1] z[0,\infty)$ 
with $y,z$ in $Y$. 
For $n\in \NN$, define 
\[
B_n\ =\ \{ y\in \pi(X'')\setminus Y:\ 
y[-n,n] \text{ is not a }Y\text{-word} \}\ , 
\]
a possibly empty wandering subset of $Y$. Because $\pi(X'')\setminus Y = \cup_n B_n$, an almost null set,
 the inclusion $Y\subset\pi(X'')$ gives an almost-Borel
isomorphism. 
Any $x\in X''$ is determined by $\pi(x)$ and $x_0$, and
 therefore $\pi|X''$ is countable-to-one. Claim (2) ensues. 


The definition of $X$ 
will depend on positive integer parameters to be specified later:
 $(n_k)_{k=1}^{\infty}$,
 $(m_k)_{k=1}^{\infty}$ and $M$. For each integer $k\geq1$ 
we add edges labeled by $2$ as follows.  
Let $\mathcal V^-_k$ and $\mathcal V^+_k$ 
be the sets of 
vertices in $G^-$ and $G^+$ corresponding to words of 
length $k$. 
For each $v_-$ in $\mathcal V^{-}_{n_k}$ and each $v_+$
in $\mathcal V^+_{n_k}$, add in an otherwise isolated extra path from 
$v_+$ to $v_-$ of 
length 
$m_k $.
We also add an extra loop based at $v_\nullset$ with
length $M$ 
(the loop is used to make the image of $\pi$ compact).

Now fix $(n_k)$ an arbitrary strictly increasing sequence of positive 
integers.  Then for large $M$ and $(m_k)$ any sequence of 
large enough positive integers, 
we have $h(X)<\epsilon$.  
For a formal proof of this (obvious) fact, one can use for example the 
 Gurevi\v{c} entropy formula, which  states 
that $h(X)$ is the growth rate of the number 
of loops based at $v_\nullset$ when their length goes to infinity. 
We choose  $\{m_1<m_2<\dots\}\cap M\NN=\emptyset$.

Define $X'''=\pi^{-1}(2^{\infty})$; $X'''$ 
is the finite orbit 
 corresponding to the special $M$ loop at $v_{\emptyset}$. 
Then (3) holds. 
Next we show $\pi$ is injective on $X'$, 
the complement of $X''\cup X'''$.  
If $y\in \pi (X')$, then 
there is at least one maximal block of $2$s in $y$ which is
bordered by a $0$ or $1$.
The length of the block 
($\infty$,  
$m_{k}$ for some $k$, or a multiple of $M$) determines 
a 
vertex in $G$ (more precisely, among the ones with ingoing or outgoing
edge labeled $2$)  from which the 
preimage of $y$ is uniquely determined.
Because all nonatomic measures on $X$ are supported 
on $X'$, Claim (1) follows, 
and also $(4)$.

%
The 
almost-Borel isomorphism claim of (5)
then follows from (1) and (2) 
 because $\pi(X)=\pi(X')\sqcup\pi(X'')\sqcup \pi(X''')$.

%


It remains to check the compactness. 
Suppose $z\in \overline{\pi(X)}$. 
If $2$ does not occur in $z$, then $z$ must be in 
$\pi(X'')$, which is compact. 
Now suppose $z=\lim \pi(x^n)$ 
for a sequence $(x^n)$, $2$ occurs in $z$ and $z\neq 2^{\infty}$.  
If a finite maximal block of $2$s occurs 
in $z$, then by considering the unique $G$-path above that block, one 
sees $z\in \pi(X')$. So suppose there is no such block.
Suppose 
$z_i\neq 2$ and $z[i+1,\infty)=2^{\infty}$. 
Let $v_n$ be the terminal vertex of 
$(x^n)_i$. If a subsequence $(v_n)$ goes to $+\infty$, 
then $z(-\infty,i]$ must be the left half of a point in $Y$; 
otherwise, a subsequence of $(v_n)$ is constant and 
$z\in \pi(X')$. The argument for 
the case $z(-\infty,i]=2^\infty$ is essentially the same. 
%
\end{proof} 
\begin{remark} 
It is an exercise to show that 
$X$ in Proposition \ref{countablebad}  
can in addition be chosen to be SPR 
(positive recurrent, and 
exponentially recurrent with respect to
its measure of maximal entropy -- see 
\cite{BBG2006} for equivalent conditions 
and reference to \cite{GurevichSavchenko} for more). 
In some ways, the SPR Markov shifts 
behave like shifts of finite type -- but not here. 
\end{remark}


\subsection{Wild Maximal Entropy}
The next result realizes a wide class of systems $T$ as equal
entropy subsystems of  continuous factors of  SFTs. 
This will be used to prove Corollary \ref{sfttopcor}.

First, we need to recall some 
definitions. 
A system is zero dimensional if its topology is generated
by clopen sets. Every such
system is topologically isomorphic to an inverse 
limit $X=X_1\leftarrow X_2\leftarrow \cdots $ 
where for all $n\in \mathbb N$, $X_n$ is 
a subshift and the bonding map 
$X_n\leftarrow X_{n+1}$ is surjective.
A continuous factor of a system is finite/zero dimensional, etc.  if as a space
it is finite/zero dimensional/etc.

The property {\it entropy-expansive} was defined 
by Bowen \cite{Bowenentexp}. 
 A zero dimensional t.d.s. is entropy-expansive
 if and only if the above inverse limit  satisfies
  $h(X)=h(X_n)$ for some $n$.
The property {\it asymptotically h-expansive} 
was a generalization defined by Misiurewicz 
\cite{Misiurewicz} 
(under the name ``topological conditional entropy'', 
which is now probably best avoided 
 \cite[Remark 6.3.18]{Downbook}). 
Any asymptotically 
$h$-expansive system has finite entropy 
and has a measure of maximal entropy
\cite{Misiurewicz}. 
The asymptotic $h$-expansiveness property 
plays an important role in the entropy theory of symbolic 
extensions \cite{Downbook}. 
A zero dimensional compact t.d.s. is asymptotically $h$-expansive 
if and only if 
 it is topologically isomorphic 
 to a subsystem of a product
$\prod_{k=1}^{\infty} X_k$ of some subshifts $X_k$ such that 
$\sum_k h(X_k) < \infty$   (see \cite{Down2001} or \cite[Theorem 7.5.9]{Downbook}).

\begin{theorem} \label{sfttop}
Suppose $T$ is a compact zero dimensional 
topological dynamical system which is  asympotically 
$h$-expansive and  is not entropy expansive. Then there is a 
continuous factor map from a 
mixing SFT onto a system 
$Y$ such that $h(T)=h(Y)$ and $Y$ contains 
a subsystem topologically 
conjugate to $T$. 
\end{theorem} 

\begin{proof}
Without loss of generality, we assume 
$T\subset X=\prod_{k=1}^\infty X_k$ where 
each $X_k$ is a mixing SFT with a fixed 
point, alphabet $\mathcal A_k$, and $\sum_k h(X_k)<\infty$. Then $X$ is a factor of a mixing SFT 
\cite[Theorem 7.1]{BoyleFiebigFiebig}. So it is 
enough to find a continuous factor 
map $\gamma: X\to Y$ such that $\gamma|T\equiv\id$, $T\subset Y\subset\prod_{k\geq1}(\mathcal A_k\sqcup\{0\})^\ZZ$, and $h(Y)=h(T)$. 

We introduce some notations. 
Suppose $R$ is a subshift 
and $M$ is a positive integer. Then 
$\mathcal W(M,R)$ is the set of words of 
length $M$ occuring in points of $R$. We let $\hXN = X_1\times \cdots \times X_N$ and $T_N$ be the projection 
of $T$ in $\hXN$. 
We write  $x\in X$ as $(x_1,x_2,\dots)$ with $x_k\in X_k$. We denote by $(x_1,\dots,x_N)|J$ the restriction of these sequences to an integer interval $J$.
Given $N,M\geq1$, $x\in X$,  we define
 $$
   I(x,N,M):=\{j\in\ZZ:(x_1,\dots,x_N)|[j,j+M)\in\mathcal W(M,T_N)\}
 $$
and let $J(x,N,M,L)$ be the union of integer intervals of length $L$ that are contained in $I(x,N,M)$. 

We shall select two non-decreasing sequences of positive integers $M_N,L_N$, $N\geq1$, and define $\gamma_N:\hXN\to (\mathcal A_N\sqcup\{0\})^\ZZ$ by:
 $$
    \gamma_N(x)=(y_j)_{j\in\ZZ} \text{ with }y_j= \alter{x_N|j  \text{ if }j\in J(x,M_N,L_N)\\
                            0  \text{ otherwise.}}
  $$
We also define $\hat\gamma_N:X\to\prod_{1\leq k\leq N} (\mathcal  A_k\sqcup\{0\})^\ZZ$ by: 
 $$
   x\mapsto (\gamma_1(x_1),\gamma_2(x_1,x_2),\dots,\gamma_N(x_1,\dots,x_N)),
 $$
and, finally, $\gamma:X\to\prod_{N\geq1} (\mathcal A_N\sqcup\{0\})^\ZZ$ by:
 $$
   \gamma(x):=(\gamma_1(x_1),\gamma_2(x_1,x_2),\dots) \text{ and let } Y:=\gamma(X).
 $$
$Y$ is a compact t.d.s.\ and a factor of $X$ and $\gamma|T\equiv\id$.

 Because $T$ is not entropy expansive, 
we have for all $N$ (perhaps after telescoping) that 
$h(T_{N+1})> h(T_N)$. Hence, we can fix a sequence of numbers $h_N$, $N\geq1$  such that $h(T_N)<h_N<h(T_{N+1})$ for all $N\geq1$.

\medbreak
It now suffices to show that there are sequences $M,L$ such that:

\medbreak
\noindent{\bf Claim.} {\it For all $N\geq1$, there is $C_N<\infty$ such that, for all $\ell\geq0$:}
 \begin{equation}\label{eq.hNbound}
     \#\{\hat\gamma_N(x)|[0,\ell):x\in X\} \leq C_N e^{h_N\ell}.
 \end{equation}
\medbreak

We extend the above claim to $N=0$, by putting $\hat \gamma_0(x):=0^\infty$, so $C_0=1$ and $h_0=0$ satisfy it for arbitrary $M_0,L_0$. We let $N\geq1$, fix $0<\eps<(h_N-h(T_N))/3$ and assume the claim for $N-1$ for some choice of $M_{N-1},L_{N-1}$.

Pick $M:=M_N\geq M_{N-1}$ such that, for some $K_1(M)<\infty$, for all $j\geq0$:
 \begin{equation}
   \#\mathcal  W(j,T_N)\leq \left(\#\mathcal  W(M,T_N)\right)^{j/M+1}\leq K_1(M) e^{(h(T_N)+\eps)j}.
 \end{equation} 

By construction, the maximal integer intervals  in $J(x,N,M,L)$ have length at least $L$. Therefore, letting $\mathcal J_\ell(N,M,L):=\{J(x,N,M,L)\cap[0,\ell):x\in X\}$, we have, for $L:=L_N$ large enough:
 \begin{enumerate}
 \item for all $\ell\geq0$, $\#\mathcal J_\ell(N,M,L)\leq K_2(L) e^{\eps \ell}$;
 \item $C_{N-1}K_1(M)\leq e^{\eps L}$.
 \end{enumerate}
 
Note that the elements of $\hat\gamma_N(x)|[0,\ell-1]$, $x\in X$, can be determined by specifying:
 \begin{enumerate}
  \item $J:=J(x,N,M,L)\cap [0,\ell)$;
  \item for each maximum integer interval $I'$ in $J$, $\hat\gamma_N(x)|I'$;
  \item for each maximum integer interval $I''$ in $[0,\ell)\setminus J$, $\hat\gamma_N(x)|I''=\hat\gamma_{N-1}(x)|I''\times 0^{I''}$.
 \end{enumerate}
For (1), the number of possibilities is bounded by:
  $$
     \#\mathcal W(\ell,Z_{L})\leq K_2(L) e^{\eps \ell}.
  $$ 
Fix one of these. Then, there are at most $\ell/L+2$ intervals $I'$ as in (2), so writing $\ell'$ for the sum of their lengths, the number of possibilities for (2) is at most:
 $$
     K_1(M)^{\ell/L+2} e^{\ell'(h(T_N)+\eps)}.
 $$
For (3), we similarly get the bound:
 $$
     (C_{N-1})^{\ell/L+2} e^{\ell''h_{N-1}}.
 $$
Thus, the number of possibilities for $\hat\gamma_N(x)|[0,\ell-1)$ is bounded by:
 $$
      K_2(L) (K_1(M) C_{N-1})^2 e^{(h(T_N)+3\eps)\ell}
 $$
As $h(T_N)+3\eps\leq h_N$, \eqref{eq.hNbound} follows for an obvious choice of $C_N$. The induction and therefore the proof is complete.
\color{black}
\end{proof}

\begin{corollary} \label{sfttopcor}
For any ergodic, finite entropy, measure-preserving system $Z$, there
is a continuous factor of a mixing SFT which admits among its
ergodic measures of maximal entropy uncountably many copies of 
 the product of $Z$ with a Bernoulli system.
\end{corollary}

\begin{proof}
Let $B=\prod_{n\geq 1}B_n$, where the $B_n$ are positive 
entropy mixing SFTs with fixed points such that 
$h(B)< \infty $. 
$B$ has a unique measure $\mu$ of maximum
entropy, the product of the unique maximum entropy measures $\mu_n$ of 
the $B_n$. 
Each $(B_n,\mu_n)$ is 
a mixing Markov chain and therefore  Bernoulli 
(by 
\cite{FriedmanOrnstein}).
It then follows from  \cite[Theorem 1]{OrnsteinProduct} 
that $(B,\mu )$ 
is also isomorphic to a Bernoulli shift 


By the Jewett-Krieger theorem, 
there is a strictly ergodic subshift $S$ which is measurably 
isomorphic to $Z$.
Let $W=S\times\prod_{n=1}^{\infty}W_n$ with  each $W_n$ 
the identity map on a two point space.
Then $B\times W$ is asymptotically 
$h$-expansive and not 
$h$-expansive  so  \Thm\ref{sfttop} applies with $T=B\times W$.
\end{proof} 

Note that the Bernoulli factor is only used to ensure the topological condition of asymptotic $h$-expansivity without entropy-expansiveness. Moreover, if
$Z$ in Cor. \ref{sfttopcor} has positive entropy and the weak Pinsker
property\footnote{ This property holds for all positive entropy ergodic systems according to the Weak Pinsker Conjecture 
\cite{Thouvenot1977,Thouvenot2008} (which remains open).} then (of course) the conclusion holds 
for $Z$ itself, with no need to  take a product with a Bernoulli system. 

The next proposition shows that the assumption that 
$T$ not be entropy expansive was  necessary for 
it to be embedded as a proper 
full entropy subsystem of a 
continuous factor of a mixing SFT.

\begin{proposition} \label{entexpprop}
Suppose $X$ is a mixing SFT, $Y$ is a 
zero dimensional continuous factor 
of $X$ and $T$ is an entropy expansive subsystem of $Y$ 
such that $h(T)=h(Y)$. Then $T=Y$.  
\end{proposition} 

\begin{proof} Let $Y$ be given as an inverse limit of 
subshifts $Y_n$ by surjective bonding maps 
$p_{n}: Y_{n+1} \to Y_{n}$.  
Let $\pi_n : Y\to Y_n$ be the projection and 
let $T_n$ be the subshift $\pi_n T$. With $p_n$ also denoting 
the restriction of $p_n$ to $T$, we have 
 $T$  as the inverse limit 
$T_n \leftarrow T_{n+1}$ 
by surjective bonding maps. Suppose $Y\neq T$. 

Pick $N$ such that $h(T_N)= h(T)$. We assume by contradiction, $T_N \neq Y_N$. Let $\gamma : X\to Y$ be the 
continuous factor map. Then $\pi_N \circ \gamma 
:= \gamma_N$ is a factor map onto $Y_N$ which is 
therefore mixing sofic. Hence $h(T_N)<h(Y_N) \leq h(Y)$, a contradiction. 
\end{proof}

\subsection{Wild period-maximal measures subsection}
We now consider the case that 
$\pi \colon X\to Y$ is a bounded to one continuous factor map from 
an irreducible SFT $X$ onto a zero dimensional system $Y$. 
In this case, $Y$ has a unique measure of maximal entropy, which 
 must be period-Bernoulli. 
If $Y$ is expansive, then $Y$ is irreducible sofic and 
almost-Borel isomorphic to a Markov shift. 
If $Y$ is not expansive then 
the Borel structure of $Y$ at a period can be very different 
from that of a Markov shift. 

Below $Y_1$ and $T_1$ denote the restrictions of $Y$ and $T$ to ergodic 
measures with maximum period 1 (see the Borel periodic decomposition \Thm\ref{thm:specdec}).

\begin{proposition} \label{badatperiod} 
Suppose $T$ is a subshift. Then there is a 
 period 2 irreducible SFT $X$ 
and a continuous factor map 
$\pi $ from $X$ onto 
a zero dimensional metrizable system 
$Y$ such that  the following hold.
\begin{enumerate} 
\item 
$|\pi^{-1}(y)| \leq 2$ , for all $y\in Y$.  
\item 
$\pi^{-1}T = \{x\in X: |\pi^{-1}(\pi (x))|=2\}$. 
\item 
$Y\setminus Y_1$ is almost-Borel isomorphic to $X$. 
\item 
$Y_1$ is almost-Borel isomorphic to $T_1$.
\end{enumerate} 
Moreover, $X$ can be chosen with $h(X)$ arbitrarily close to $h(T)$.
\end{proposition} 
                     
\begin{proof}  
We choose $(X,\sigma )$ of the form 
$X=X'\times (\ZZ/ 2\ZZ)$, with 
$\sigma : (x,g)\mapsto (\sigma x, g+1)$, 
where 
$(X',\sigma )$ is any mixing SFT into which 
$T$ continuously embeds with entropy arbitrarily close to $h(T)$.  Let $E'$ be the quotient relation 
of the map $T\times \ZZ/ 2\ZZ \to T$ defined 
by $(x,g)\mapsto x$. Let $E$ be the union of 
$E'$ and the diagonal of $X$. Define $Y$ as 
the quotient space $X / E$ (with quotient 
topology) and identify the image 
in $Y$ of $T\times \{0,1\}$ with $T$. 
Then $Y$ is compact metrizable, 
since $E$ is a closed equivalence relation
 (Proposition \ref{compactrelation}).  
Let us check that $Y$ is zero-dimensional.
For an $X'$ word $W_{-n}\dots W_n$, 
let $U_w=\{x\in X': x[-n,n]=W\}$. If $W$ 
is not a $T$-word, then $\pi U_W$ is clopen 
in $Y$; if $W$ is a $T$-word, then 
$\pi (W\times \ZZ / 2\ZZ)$ is clopen in $Y$. 
Therefore each point in  $Y$ has a neighborhood 
basis of clopen sets. 

The system $X'\setminus T$ contains 
mixing SFTs with entropy arbitrarily
close to $h(X)$. Hence $Y\setminus Y_1$ is 
the union of a 
 strictly 
$(h(X),2)$-universal Borel system and a  
period-2 Bernoulli measure of 
entropy $h(X)$.  Therefore $Y\setminus Y_1$ 
is almost-Borel isomorphic to $X$. The rest  is clear. 
\end{proof} 

We'll give two easy corollaries of Proposition \ref{badatperiod} 
which already show $Y_1$ can be very different from 
what can arise in a Markov shift. 

\begin{corollary} \label{badbyjk} 
Suppose $(W,\nu )$  is a totally ergodic, finite entropy, 
mea\-sure-pre\-serving system. 
 Then 
there is a period 2 irreducible SFT $X$ and 
a continuous, at most 2-to-1 factor map $\pi\colon X\to Y$
such that $Y_1$ is almost-Borel isomorphic to 
$(W,\nu)$.
\end{corollary} 
\begin{proof} 
This follows from \Prop\ref{badatperiod} 
and the Jewett-Krieger Theorem. 
\end{proof}  

Let $R$ be the map on $\mathbb T^2 $ defined 
by $(t,y)\mapsto (t,y+t)$. 
Let $ P_0 = \{ (x,y)\in \mathbb T^2\colon 
0\leq x \leq y  \leq 1
\}$ and $P_1= \mathbb T^2 \setminus P_0$.
Let $Z$ be the subshift 
on symbols $0,1$ which is the 
closure of $R$-itineraries   through 
the partition $\{P_0 ,P_1 \}$. 
$Z$ is a disjoint union of Sturmian shifts (one for each 
irrational rotation) and countably many periodic orbits. 
Now $Z_1$ is the restriction of $Z$ to the complement of the 
periodic orbits of period greater than 1 (including exactly 
one copy of each Sturmian shift and a fixed point). 

\begin{corollary} \label{badfinite}
Suppose $(W,\nu )$  is a weakly mixing, finite entropy, 
ergodic transformation. 
There is a period 2 irreducible SFT $X$ and 
a continuous at most 2-to-1 factor map $\pi\colon X\to Y$, 
such that 
$Y_1$ is almost-Borel isomorphic to 
$Z_1 \times (W,\nu )$. In particular, the measures of $Y_1$
are uncountably many and have entropy $h(W)$.
\end{corollary} 
\begin{proof} 
By the Jewett-Krieger 
Theorem, let $W'$ be a strictly ergodic 
shift, which with its  invariant measure 
 is isomorphic to $(W,\nu)$. 
Set $T$ in \Prop\ref{badatperiod} 
to be $Z\times W'$. A product of irrational rotation (or fixed point) and 
weakly mixing remains totally ergodic so $Y_1$ and $T_1$ are
isomorphic to $Z_1\times W'$.
\end{proof}  

Obviously,  the possible almost-Borel structure of $Y_1$ in 
\Prop\ref{badatperiod}  can be much more 
varied than shown in the two corollaries.

\section{$C^{1+}$ surface diffeomorphisms} \label{sec:surface} 

\subsection{Sarig's Symbolic Dynamics}

For each compact surface $C^{1+}$-diffeo\-mor\-phism $f:M\to M$ and 
number $\chi>0$,
Sarig \cite{Sarig} defined
$\pih,\Sigmah,\Sigmah^\#, \mathcal R, \sim$ such that
$\Sigmah$ is a Markov shift with countable alphabet 
$\mathcal R$;   
$\pih$ is a Borel factor map from 
$\Sigmah$ into $M$; and there is a relation 
on the elements of $\mathcal R$  of being ``affiliated'' 
(which we will write as $\sim $). We note that
$\Sigmah^\#$ (the \lq\lq regular set\rq\rq) 
is the \sarig\ set $\Sigmah_{\pm \text{ret}}$ of  
Definition \ref{defnxret}. 

\begin{summary} \label{summary} 
The items above satisfy the following.  
\begin{enumerate}
 \item 
If $\mu \in \pe (f)$ and has both its positive and 
negative Lyapunov exponents  outside $(-\chi,\chi)$, 
then  $\mu \pih(\Sigmah^\#)=1$.
 \item 
If $\mu \in \pe (f)$ and 
$h(f,\mu) \geq \chi$, 
then  $\mu \pih(\Sigmah^\#)=1$.
 \item Each point $z\in\pih(\Sigmah^\#)$ has only finitely many
   preimages in $\Sigmah^{\#}$.
\item 
$\pih$ is Bowen type on $\Sigmah^\#$ for the relation $\sim$ (see Defn.\ \ref{def.AlmostBowen}).
\item 
For all $R\in\mathcal R$, $\{R'\in\mathcal R:R'\sim R\}$ is finite.
\item 
$\pih$ is H\"older-continuous. 
\item 
$\hat\Sigma$ is locally compact.  
\end{enumerate}
\end{summary} 
This symbolic dynamics is an embarassment of 
riches.
To apply 
Theorem \ref{bowentypefactors}, 
we only need that  
$\pih$ is finite-to-one 
Bowen type on $ \Sigmah^\#$, 
which follows from 
(3,4). Properties (5,6,7) are given for context. 

Properties (1,2) are of course 
essential to relating the symbolic 
dynamics to the diffeomorphism. We note that the main theorems of
\cite{Sarig} quote property (2). 
This is weaker than (1): as is
well-known (see \cite{Katok1980}), for a  surface diffeomorphism, an
ergodic measure with nonzero entropy
must have no zero
Lyapunov exponent. 
However the proofs deal with the set
$\operatorname{NUH}_\chi(f)$ which is defined \cite[p. 348]{Sarig} in
terms of the exponents, not the entropy, which is never used in the
rest of the paper.\footnote{The author has confirmed to us that 
the \emph{remark on $\chi$-largeness}
\cite[p.344]{Sarig} 
contains a misstatement:  there, 
\lq\lq both Lyapunov exponents\rq\rq\
  should replace 
\lq\lq at least one Lyapunov
exponent\rq\rq.}

We will see below that 
the properties in the summary 
are explicitly or 
essentially contained 
in \cite{Sarig}. 

\subsection{The theorem for surface diffeomorphisms}

\begin{theorem} \label{surfacetheorem}
Every $C^{1+}$ surface diffeomorphism 
$(X,f)$ is the union of two Borel subsystems $Y$ and $Z$ such that:
 \begin{itemize}
  \item $Y$ is almost-Borel isomorphic to a Markov shift;
  \item $Z$ carries only zero entropy measures. 
 \end{itemize}

Moreover,  a
nonatomic ergodic measure is carried by $Z$  if
  and only if it satisfies all of the following conditions:
 \begin{enumerate}[(i)]
  \item its entropy is zero;
  \item at least one of the Lyapunov exponents is zero;
  \item it has no period which is the maximal period of an ergodic, invariant probability with positive entropy.
 \end{enumerate}
\end{theorem}

\begin{remark}
The conditions (i)-(iii) are not independent. As discussed above, (ii)
implies (i). 
Also (iii) is equivalent to:
 \begin{enumerate}
  \item[(iii')] the measure has no period which is the maximal period
    of a nonatomic, ergodic, invariant probability which has no 
zero Lyapunov exponent. 
 \end{enumerate} 
\end{remark}

\begin{remark}\label{hyperbolicremark}
Note that the ``universal'' part of $Y$ above 
could alternately be argued from Corollary \ref{cor.manySFTs} and Katok's horseshoes (see \cite{BuzziLausanne}, where this is
done in any dimension, assuming no zero Lyapunov exponents). But to 
control measures with entropy maximal at a period, we depend 
 on Sarig's symbolic 
dynamics. 
\end{remark} 


\begin{proof}[Proof of \Thm\ref{surfacetheorem}] 
For $\chi=1/n$, we apply Sarig's work to get a Markov shift 
$\Sigmah_n$ and  factor map 
$\pih_n:\Sigmah_n\to X$ 
satisfying \ref{summary}(1-4). 
Let $\overline{\Sigma}_n$ be the union of the Sarig regular sets of all irreducible components of $\Sigma_n$.
By properties  \ref{summary}(3,4) and 
\Thm\ref{bowentypefactors}, 
$\hat{Y}_n:=\pih_n (\overline{\Sigma}_n)$ 
 is
 almost-Borel
isomorphic to a Markov shift.
Let $Y_0=\cup_n \hat{Y}_n$; 
by \Cor\ref{markovunion}, $Y_0$ is almost-Borel isomorphic to a 
Markov shift.


If $\mu\in\pen(f)$ satisfies neither (i), nor (ii), then, by properties \ref{summary}(1,2), there exists $\hat\mu\in\pen(\Sigmah_n)$ with $\pih_*(\hat\mu)=\mu$. In particular, $\hat\mu(Z)=1$ for some irreducible component of $\Sigmah_n$, so $\hat\mu(Z_{\pm\ret})=1$, and therefore $\mu(Y_0)=1$.
We enlarge $Y_0$ into
  $Y$ carrying all measures not satisfying 
all of (i)-(iii) as follows. 

First, let
  $\lambda^u(x):=\limsup_{n\to\infty}\frac1n\log\|(f^n)'(x)\|$. It is
  a Borel function such that, for all $\mu\in\Proberg(f)$, for
  $\mu$-a.e. $x\in X$, $\lambda^u(x)$ is the largest exponent of
  $\mu$. By this observation (and the same applied to the smallest
  exponent), we get an invariant  Borel subset $X''$ which has full
  measure for $\mu\in\Proberg(f)$ if and only if $\mu$ has a zero
  Lyapunov exponent.

Now let $P$ be the set of integers $p\geq 1$ 
such that there is some  ergodic, invariant probability measure $\mu$
with nonzero entropy with maximal period $p$. 
For each $p$ in $P$, $\Sigma$ contains an irreducible Markov shift
 $\Sigma_p$ with some period dividing $p$ and positive entropy, and  
therefore $u_{Y_0}(p)>0$. 
For each $p\in P$, the Borel periodic decomposition (\Thm\ref{thm:specdec})  
provides an invariant  Borel subset
$X'_p$ of $X$ such that for $\mu \in \pe (X)$, 
$\mu (X'_p) =1$ if and only if $p$ is a period of $\mu$. 
Define $Y_p:=X'_p\cap X''$ and 
 $Y:=Y_0\cup\bigcup_{p\in P} Y_p$. 
Because all measures on $Y_p$ have zero
 entropy and $u_{Y_0}(p)>0$ for $p$ in $P$, 
by
Corollary \ref{useqtheorem}(\ref{absorbing})
the systems $Y$ and $Y_0$ are almost-Borel isomorphic. 

Thus $X=Y\sqcup Z$, with $Z:=X\setminus Y$, is an invariant, Borel decomposition such that $Y$ satisfies (1) and (2) and carries any $\mu\in\pen(f)$ failing to satisfy one of (i),(ii),(iii). Conversely,  $\mu(Z)>0$ implies (i), (ii), and $\mu(Y_p)=0$ for all $p\in P$, hence (iii).
\end{proof}

As an invariant, ergodic probability measure with trivial rational spectrum has maximal period equal to $1$, this yields:

\begin{corollary}
Consider a positive entropy, $C^{1+}$ diffeomorphism of a compact surface. 

It is almost-Borel
isomorphic to a  Markov shift if it has a totally ergodic  measure
with positive entropy. 

It is  almost-Borel isomorphic to a mixing Markov shift if it has a totally ergodic measure which is the 
unique measure of maximum entropy. 
\end{corollary}

\begin{remark}
The situation of the corollary occurs in some natural settings. In particular, Berger \cite{Berger} has shown that for a positive Lebesgue measure subset of parameters, H\'enon maps have a unique measure of maximal entropy that is mixing. Their invariant measures are carried by a forward invariant compact disk and therefore one can apply the above corollary: these H\'enon maps are almost-Borel isomorphic to a mixing Markov shift. In particular, they are $h$-universal, where $h$ is their Borel entropy (equal to their topological entropy after restricting to the invariant disk).
\end{remark}

%

\subsection{Proof of the properties of Sarig's construction}
We now discuss how the Summary \ref{summary} properties 
come from Sarig's paper. For (1,2,3,6,7),
see  \cite[Theorems 1.3, 12.5, 12.8]{Sarig}. 
Property  (5) is a statement within the proof of Lemma 12.7. 
To explain (4), 
we need some facts and notations from Sarig's paper \cite{Sarig}.

\subsubsection*{The set $\mathcal V$ of Pesin charts and the Markov shift $\Sigma(\mathcal G)$}

Sarig builds a countable collection $\mathcal V$ of triplets $(\Psi_x,p^s,p^u)$ where $p^s,p^u>0$ and $\Psi_x$ is a Pesin chart defined using the Oseledets theorem applied at point $x$. Charts are diffeomorphisms on their image with Lipschitz constant at most $2$ and the domain of $\Psi_x$ contains $(-p^s,p^s)\times(-p^u,p^u)$. We often write $p$ for $\min(p^u,p^s)$ and, following Sarig, write the triplet as $\Psi_x^{p^s,p^u}$ and continue to call it a chart (despite the extra information $p^u,p^s$). 

Sarig defines a graph $\mathcal G$ over $\mathcal V$.  In particular,
$\Psi_x^{p^s,p^u}\to\Psi_y^{q^s,q^u}$ in $\mathcal G$ implies that, at
least on the rectangle $(-10p,10p)$, $f_{x,y}:=\Psi_y^{-1}\circ f\circ\Psi_x$ is uniformly hyperbolic and $\Psi_y^{-1}\circ\Psi_x$ is very close to 1. More precisely, for $(u,v)\in(-p^s,p^s)\times(-p^u,p^u)$
 $$
  f_{x,y}(u,v)=  (A_{x,y}u,B_{x,y}v)+ h(u,v)
 $$
with $C_f^{-1}<|A_{x,y}|<e^{-\chi}, e^{\chi}<|B_{x,y}|<C_f$ and
$\| h(0) \| \leq\eps q$ and $\|h'(0)\|\leq 2\eps p^{\beta/3}<\eps$ (see \cite[Prop. 3.4, p.14]{Sarig}).

It follows that, for any sequence $\underline{v}=(\Psi_{x_n}^{p_n^s,p_n^u})_{n\in\ZZ}\in\Sigma(\mathcal G)$, there is a unique sequence $\underline{t}\in(\RR^2)^\ZZ$ such that
 $$f_{x_n,x_{n+1}}(t_n)=t_{n+1}\in B(0,p_{n+1})
 $$
for all $n\in\ZZ$. The projection $\pi:\Sigma(\mathcal G)\to M$ defined by Sarig \cite[Proposition 4.15, Theorem 4.16]{Sarig} satisfies: $\pi(\underline{v})=\Psi_{x_0}^{p^u,p^s}(t_0)$ and $t_n\in B(0,p_n/100)$ for all $n\in\ZZ$.

According to \cite[Theorem 5.2]{Sarig}, if
$\pi(\underline{v})=\pi(\underline{w})$ for
$\underline{v},\underline{w}\in\Sigma(\mathcal G)^\#$, then, for each
integer $n\in\ZZ$, the charts $v_n=\Psi_{x_n}^{p_n^u,p_n^s}$
 and $w_n=\Psi_{y_n}^{q_n^u,q_n^s}$ are very close:
 on $B(0,\eps)$ ($\eps$ is much larger than $p,q$, see \cite[Def.2.8
 and Lem.2.9]{Sarig}) 
  \begin{equation}\label{eq.closePsi2}
    \Psi_{y_n}^{-1}\circ\Psi_{x_n}(t) = \pm t+\delta(u) \text{ where }
    \|\delta(0)\|<q_n/10,\; \|\delta'\|\leq \eps^{1/3}.
 \end{equation}

\subsubsection*{Cover $\mathcal Z$ by large rectangles}

Sarig then defines a cover:
 $$
   \mathcal Z:=\left\{ Z(v):v\in\mathcal V\right\} \text{ with }Z(v):=\{\pi(\underline{v}):\underline{v}\in\Sigma(\mathcal G)^\#,\; v_0=v\}
 $$
Proposition 4.11 of \cite{Sarig} implies that $\Psi_x^{-1}(Z(v))\subset B(0,q/100)$, well inside the domain of the chart.
 
\subsubsection*{Partition $\mathcal R$ by small rectangles}
Sarig refines the cover $\mathcal Z$ into a \lq\lq Markov
partition\rq\rq\ $\mathcal R$, following an elaborate version of the
Bowen-Sina\u{\i}  
 construction used in the uniformly hyperbolic case. $\Sigmah$ is then the Markov shift defined by the countable oriented graph with vertices $R\in\mathcal R$ and arrows $(R,R')\in\mathcal R^2$ if and only if $f(R)\cap R'\ne\emptyset$. The map $\pih:\Sigmah\to M$ satisfies:
  $$
     \{\pih((R_n)_{n\in\ZZ})\}=\bigcap_{n\in\ZZ} f^{-n}(\overline{R_n}) = \bigcap_{n\in\ZZ} f^{-n}(\overline{Z_n}) 
  $$
for some $Z_n\in\mathcal Z$, $Z_n\supset R_n$.

\subsubsection*{Affiliated small rectangles}

Sarig defines two small rectangles $R,R'\in\mathcal R$ to be {\bf affiliated} (see before Lemma 12.7 in \cite{Sarig}) when there are two large rectangles $Z,Z'\in\mathcal Z$ such that:
 $$
     R\subset Z,\; R'\subset Z' \text{ and } Z\cap Z'\ne\emptyset.
 $$

\begin{proof}[{\bf Proof of \ref{summary}(4)}]
Claim 2 in the proof of Theorem 12.8 in \cite{Sarig} asserts precisely that, for $R,R'\in\Sigmah$, if $\pih(R)=\pih(R')\in\pih(\Sigmah^\#)$ then $R_n$ and $R_n'$ are affiliated for each $n\in\ZZ$. Thus, it suffices to prove: for all $R,R'\in\Sigmah$, if $R_n$ and $R_n'$ are affiliated for each $n\in\ZZ$, $\pih(R)=\pih(R')$. Let $x=\pih(R)$, $y=\pih(R')$. For each $n\in\ZZ$, writing $Z_n=Z(\Psi_{x_n}^{p_n^s,p_n^u})$,
 $$
   f^nx\in\overline{R_n}\subset\overline{Z_n}\text{ and }
     t_n:=\Psi_{x_n}^{-1}(f^nx)\in\Psi_{x_n}^{-1}(Z_n)\subset B(0,p_n/100).
 $$ 
Likewise,
 $$
    u_n:=\Psi_{y_n}^{-1}(f^ny)\subset B(0,q_n/100).
 $$
Now, using $q_n\leq e^{\eps^{1/3}}p_n$ and eq.  \eqref{eq.closePsi2}, we get, for all $n\in\ZZ$,
 $$
     u'_n:=\Psi_{x_n}^{-1}\circ\Psi_{y_n}(u_n)\in B(0,p_n/10+(1+e^{\eps^{1/3}})p_n/100)\subset B(0,p_n)
 $$
so $u'_{n+1}=F_n(u'_n)$ where $ F_n:=\Psi_{x_{n+1}}^{-1}\circ f
\circ\Psi_{x_n}$. The uniform hyperbolicity of these maps on their
domains $B(0,p_n)$ implies that $u'_n=t_n$ for all $n\in\ZZ$. In
particular, $x=y$. 
\end{proof}



\subsection{Classification from measures of given maximum period}\label{sec.proofSurfClass}

\begin{proof}[Proof of Theorem \ref{mainthm.surfclass}]
Isomorphic diffeomorphisms have equal data (1) and (2), since those only depend on positive entropy measures. We turn to the converse.
By Theorem \ref{mainthm.conjug}, it suffices to classify the isomorphic Markov shifts up to almost-Borel isomorphism. By Theorem \ref{mainthm.classif}, it suffices to show that the data (1) and (2) are equal to $\bar u_S(\cdot)$ and $\bar \eta_S(\cdot)$ for any isomorphic Markov shift $S$. We fix $p\geq1$ and use Fact \ref{f.MSmaxper}.

First the Fact implies that $\bar u_S(p)$ is indeed equal to the supremum in (1). Second, let $\mathcal M(p)$ be the measures counted in (2) and $\mathcal S(p)$ be the irreducible subshifts counted by $\bar\eta_S(p)$. Associate to any $\mu\in\mathcal M(p)$ the irreducible shift $\Sigma_i$ carrying its image in $S$. 

The Fact implies  $p_i|p$, hence $h_i\leq \bar u_S(p)$ so $\mu$ is a
\mme\ of $\Sigma_i$. Thus $p_i=p$ and $\Sigma_i\in\mathcal S(p)$.
Since the \mme\ of $\Sigma_i$ is unique, $\mu\mapsto\Sigma_i$ is
injective. Conversely, for any $\Sigma_i\in\mathcal S(p)$, (the image
on the surface of) its \mme\ belongs to $\mathcal M(p)$. Hence,
$\mu\mapsto\Sigma_i$ is a bijection and $\#\mathcal
M(p)=\bar\eta_S(p)$.
\end{proof}

\section{Open problems} \label{sec:summary}

We  select and discuss a few  open problems.
Observe that the universality results in this paper and \cite{Hochman} 
address only systems with topological embeddings of positive entropy SFTs (often as the consequence of hyperbolicity).
However, the following result of Quas and Soo 
suggests that this strong kind of hyperbolicity is 
not necessary for Borel universality.

Recall that a toral 
automorphism arising from matrix $A$ is quasi-hyperbolic if $A$ has
an irrational eigenvalue on the unit circle \cite{Lind1982}.
It is {\it irreducible} if the characteristic 
polynomial of $A$ is irreducible. Lindenstrauss and Schmidt  \cite{LindenstraussSchmidt2004}
showed that  irreducible quasihyperbolic toral automorphisms cannot contain 
nontrivial homoclinic points, and 
therefore cannot contain (or be a continuous factor of) 
any positive entropy SFT.

Nevertheless, Quas and Soo  \cite{QuasSoomaps} have proven an analogue of the Krieger generator theorem (which is the starting point of Hochman's result) for this class. This generalization raises the following:

\begin{pbm} \label{quasi} Suppose $(X,T)$ is a 
mixing quasihyperbolic toral 
automorphism\footnote{More generally, the question can be asked about the class of maps considered by \cite{QuasSoomaps}: compact t.d.s.\ that satisfy almost weak specification, asymptotic entropy expansiveness, and the small boundary property.}.  Must $(X,T)$ be $h(T)$-universal (as in 
Theorem \ref{thm.Hochman})? 
\end{pbm} 

A different question related to the absence of hyperbolicity is:

\begin{pbm} \label{zerosurface}  Complete the almost-Borel classification of $C^{1+}$ surface 
diffeomorphisms (i.e., extend Theorem \ref{mainthm.conjug} to 
address all nonatomic, ergodic measures). 
\end{pbm} 


In another direction, our proofs require $C^{1+}$-smoothness (for the application of Sarig's \cite{Sarig} symbolic dynamics and ultimately Pesin theory  \cite{Pugh1984,BonattiCrovisierShinohara2014}). Rees' examples \cite{Rees} (see also \cite{BCL2012} and references therein) show that our results do not extend to homeomorphisms.

\begin{pbm}\label{C1} 
Are $C^1$ surface
diffeomorphisms Borel isomorphic to Markov shifts  away from zero entropy measures? 
In positive topological entropy, can they have ergodic period-maximal
measures that are not period-Bernoulli, or have uncountably many 
ergodic period-maximal measures?
\end{pbm} 

Finally, in light of Theorem \ref{mainthm.conjug}, we 
ask the following. 
 
\begin{pbm} \label{markovsurface} 
Which Markov shifts of finite positive 
entropy can be almost-Borel isomorphic to 
a $C^{1+}$ surface diffeomorphism? 
\end{pbm} 

We are not able to rule out the possibility that 
every Markov shift of finite positive entropy is
almost-Borel isomorphic to a surface diffeomorphism.

\appendix

\section{Borel periodic decomposition}\label{sec:BPD}

This Appendix provides a proof of \Thm\ref{thm:specdec}. We freely use the notations of the Theorems and definitions and facts from Sec. \ref{periodssubsec}. We assume $p\geq2$, the case $p=1$ being trivial.

The space 
of finite measurable partitions
 of $X$ into $p+1$ atoms is: 
 $$
  \mathcal P
 = \{ (P_1,\dots ,P_p,P_{p+1}): 
P_i \text{ is Borel; } P_i\cap P_j=\emptyset \text{ if }i\neq j ;\ 
\cup_i P_i=X\}.
$$    

If $C:=(C_1, \dots , C_p)$ is a $p$-cyclic partition for some measure 
$\mu\in\mathcal M$, set $\hat C:=(\hat C_1,\dots,\hat C_p,X\setminus \cup_i \hat C_i)$ where 
 $$C'_i = C_i\setminus \cup_{ j\neq i}C_j \text{ and }
     \hat C_i = C'_i \cap (\cap_{n\in \ZZ} T^n (\cup_{j=1}^p C'_j)),$$ 
so $\hat C\in\mathcal P$. 
Moreover, $\mu(\hat C_i\Delta C_i)=0$ and
$T(\hat C_i)=\hat C_{i+1}$ (again $\hat C_{p+1}=C_1$) for all
$i=1,\dots,p$ and $(\hat C_1,\dots,\hat C_p)$ is still a $p$-cyclic
partition for
$\mu$.

Finally each $\mu\in \PP (X)$ defines a pseudometric $\rho_{\mu}$ 
on $\mathcal P$: 
$
\rho_{\mu}(P,Q)= \frac 12 \sum_{j=1}^{p+1} \mu(P_j \sd Q_j)\   .
$
We will appeal to the following theorem of  Kieffer and Rahe.

\begin{thm}\cite[Thm. 5]{Kieffer-Rahe}\label{krahe}
 Let $\mathcal D$ be a Borel subset of $\pe(T)$ and let 
$\{ \mathcal P_{\mu}: \mu \in \mathcal D\}$
 be a collection of nonempty subsets of $\mathcal P$ such that 
\begin{enumerate} 
\item each $\mathcal P_{\mu}$ is $\rho_{\mu}$-closed, and 
\item
for each $P$ in $\mathcal P$, the map 
$\rho_P: \mathcal D \to [0,1]$ defined by 
$\mu \mapsto \inf \{ \rho_{\mu}(P,Q): Q\in \mathcal P_{\mu} \}$ 
is Borel measurable. 
\end{enumerate} 
Then $\cap_{\mu} \mathcal P_{\mu} \neq \emptyset $ .
\end{thm} 

\begin{proof}[Proof of \Thm\ref{thm:specdec}] 
Let $\mathcal D=\{
\mu \in \pe(T): 
e^{2i\pi/p}\in \sprat(T,\mu )
\}$.

Given $\mu\in \mathcal D$, let $\mathcal P_{\mu}$ be 
the set of $\hat C\in\mathcal P$ for all $p$-cyclic partitions $C$ for $\mu$.
It remains to show 
 $\cap_{\mu}  \mathcal P_{\mu} \neq \emptyset $ .
Note, 
each $\mathcal P_{\mu}$ is $\rho_{\mu}$-closed, so 
condition (1) of Theorem \ref{krahe} is satisfied. 

Given $\mu \in \mathcal D$, there are distinct $\nu_i$ in 
$\pe(T^p)$, $1\leq i\leq p$, such that 
$\mu = \frac 1p \sum_i \nu_i$ and 
$T\nu_i = \nu_{i+1}$, 
$1\leq i \leq p$ ($\nu_{p+1}$ means $\nu_p$). 
Given $\mu$, let $C_1, \dots , C_p$ be disjoint sets such that 
$\nu_i(C_i)=1$, $1\leq i \leq p$. 
Observe that the ergodicity of $\mu$ implies that elements of
$\mathcal P_\mu$ coincide modulo $\mu$ up to a cyclic permutation of their
first $p$ elements. Thus, modulo $\mu$, $\mathcal P_{\mu}$ contains 
exactly $p$ elements,  the cyclic permutations $(C_{1+d}, \dots , C_{p+d},C_*)$,
$d=0,\dots,p-1$. 

To check that $\mathcal D$ is a Borel subset of the Borel set 
$\pe(T)$, we appeal to some background facts. 
An  injective Borel measurable map into a Borel space has a Borel image, 
and  a Borel measurable inverse
\cite[(15.2)]{Kechris}. 
The fixed point set of a Borel automorphism is Borel. 
For $E$ a separable metric space, the Borel field 
of  $\PP (E)$ (and hence of any Borel subset of 
$\PP (E)$ is the smallest field for which the maps  
$\mu\mapsto \mu(A)$, $A$  ranging over the Borel sets of $E$, 
are measurable 
\cite[Theorem 17.24]{Kechris}.  
Consequently, 
the sets $F_i, G_1,G_2,G_3$ below are Borel:  
$$\begin{aligned} 
F_i \ =\  \{\mu \in \PP (T^i): T^i\mu = \mu \} \quad \ \ \ & \quad 
G_1 \ =\  \pe (T^p)\setminus \cup_{i=1}^{p-1}F_i \\
G_2\ =\  \{ \frac 1p \mu : \mu \in G_1\}  
\qquad \qquad \quad \ 
 & \quad 
G_3 \ =\  \Bigg\{ \sum_{i=1}^pT^i \mu: \mu \in G_2\Bigg\}  & 
\ . 
\end{aligned} $$
We claim that $\mathcal D=G_3$. 
If $\nu\in \mathcal D$ and $\gamma $ 
is the assumed factor map onto $\{e^{2\pi i/k}: k = 0, 1, \dots, 
p-1\}$, let  $\mu $  be $p$ times the restriction of $\nu $ 
to $\gamma^{-1}(1)$. Then $\mu \in G_1$ (because 
$\mu$ is ergodic for $T$) and $\nu = \sum_{i=1}^{p}\mu$. 
Therefore  
$\mathcal D$ is contained in  $G_3$. For the other direction, 
suppose $\mu \in G_1$.  Given
 $1\leq i \leq p-1$, 
write the measure $T^i\mu$ as $\nu_c + \nu_s$, 
where $\nu_c = f\mu$ ($f$ the Radon-Nikodym derivative) 
and $\nu_s$ is singular with respect to $\mu$. 
The function $f$ is $T^p$-invariant, because the 
measures $\mu$ and $T^i\mu$ are $T^p$-invariant, 
so by ergodicity of $\mu$ for $T^p$, $f$ is constant 
$\mu $ a.e.
Because $T^i\mu \neq \mu$,  there is then a 
set $C_i$ of $\mu$-measure 1 and $T^i\mu$-measure zero. 
Let $C=\cap_{i=1}^{p-1}C_i$ and $D_i= T^iC$, $0\leq i \leq p-1$. 
It follows that $\mu (D_i\cap D_j)=0$ for $0\leq i < j \leq p-1$. 
Now $\sum_{i=0}^{p-1}\frac 1p T^i\mu$ is a $T$ invariant probability 
eigenfunction defined a.e. by  
$x\to e^{2\pi i/p}$ if $x\in D_i$.  Therefore 
$G_3$ is contained in 
$\mathcal D$. 

It remains to verify condition (2) of Theorem \ref{krahe}. 
We will construct a Borel selection $\beta$  for the Borel map 
$\phi:G_2\to \mathcal D$ defined by $\nu \mapsto \sum_{i=1}^p 
T^i\nu$ (i.e., $\beta: \mathcal D \to G_2$ is Borel 
and $\phi \circ \beta$ is the identity on $\mathcal D$). 

Define a Borel measurable order $\prec$ on $G_2$ 
(for example, via a Borel injective map $G_2 \to \mathbb R$). 
 Let $B=\{ m\in G_2: m\prec T^jm , 1\leq j < p\}$, a Borel set in
 $G_2$.  Then the restriction  $B\xrightarrow[ ]{ \phi }\mathcal D$ is
 a Borel bijection and  
$\beta 
=  (\phi |{B} )^{-1}$ is our selection.

Now suppose $P=(P_1, \dots ,P_{p+1}) \in \mathcal P$.
For $\mu \in \mathcal D$, 
set $\mu'=\beta (\mu)$.
Given 
$Q=(Q_1, \dots , Q_{p+1})$ in $\mathcal P_{\mu}$, 
there is some $d\in \{ 1, \dots ,p\}$ 
such that 
for $1\leq j \leq p$
we have 
\begin{align*}
(T^{j+d}\mu')(Q_j)\ &=\ \mu (Q_j) \ ,\\
(T^{j+d}\mu')(X\setminus Q_j)\ &=\ 0\ , 
\end{align*}
and $\mu (Q_{p+1})=0$. Therefore 
\begin{align*} 
\rho_\mu(P,Q)\ =\ 
\frac 12 \sum_{j=1}^{p+1} \mu (P_j \sd Q_j) 
\ &= \ 
\frac 12 \sum_{j=1}^{p+1} \mu (P_j) + \mu (Q_j) -\mu (P_j \cap Q_j) \\ 
= 1 - \frac 12 \sum_{j=1}^p \mu (P_j \cap Q_j) 
\ &= \ 1 - \frac 12 \sum_{j=1}^p (T^{j+d}\mu ')(P_j) 
\ := \ \phi_d (\mu ) \ .
\end{align*}
We conclude that 
$\inf\{\rho_\mu(P,Q): Q\in \mathcal P_{\mu}\} = 
\min \{\phi_d(\mu ): 1\leq d\leq p\}$, 
which is a Borel function of $\mu$. 
\end{proof}

\section{Miscellany} 

We include in this section some basic results 
for lack of a direct reference. 

\begin{proposition} \label{finitetoonelift} 
Let $\pi:(X,S)\to(Y,T)$ be a Borel factor map. Let $\nu\in\Prob(T)$ satisfy: for $\nu$-a.e. $y\in Y$, $0<\#\pi^{-1}(y)<\infty$. Then there exists $\mu\in\Prob(S)$ such that $\pi_*\mu=\nu$.
\end{proposition}

\begin{proof}
Observe that we can replace $Y$ by $\bigcap_{n\in\ZZ}T^{-n}Y'$ where
$Y'$ is a Borel set of full $\nu$-measure implied by the assumption. 

We claim that there are a Borel map $N:Y\to\NN$, $N(y):=\#\pi^{-1}(y)$, and a Borel isomorphism  $\psi:X\to\hat Y:=\{(y,k)\in Y\times\NN:1\leq k\leq N(y)\}$ such that $\pi\circ\psi(y,k)=y$ on $\hat Y$. This follows from the uniformization theorem for Borel maps with countable fibers \cite[(18.10) and (18.14)]{Kechris}.

Now, $\psi\circ S\circ\psi^{-1}(y,k)=(T(y),\sigma_y(k))$
where $$\sigma_y:\{1,\dots,N(y)\}\to\{1,\dots,N(Ty)\}.$$ $S$ and $T$
being automorphisms, $N\circ T=N$ and $\sigma_{y}$ is a
permutation of $\{1,2,\dots,N(y)\}$. Hence, $S$ must preserve
 $$\mu:=\sum_{n\geq1} 
(\psi^{-1})_*\left((\nu|N^{-1}(n))\times\tfrac1n(\delta_1+\dots+\delta_n)\right).$$
\end{proof}

\begin{proposition} \label{compactrelation}
Suppose $f: X\to Y$ is a continuous surjection, 
$Y$ has the quotient topology, 
$X$ is compact metric and 
$E:=\{ (x,w): f(x)=f(w)\}$ is closed in $X\times X$.  
Then $Y$   is compact metrizable. 
\end{proposition} 
\begin{proof}  
Let $p_1, p_2$ be the projections from 
$X\times X $ to $X$. 
If   $K$ is a closed subset of 
the compact Hausdorff space $X$, then  
$f^{-1}(f(K))=\pi_2(\pi_1^{-1}K))$ is 
closed in $X$. Now $f$ is a closed map 
with compact fibers and $X$ is metrizable,
so $Y$ is metrizable 
\cite[Theorem 5.2]{Dugundji}.
\end{proof} 


\bibliographystyle{plain} 
\bibliography{b}

%
%
%
%
%
%
%
%
%
%

\end{document}